\documentclass[a4paper]{amsart}
\usepackage{amsmath,amsthm,amssymb,latexsym,epic,bbm,comment,mathrsfs}
\usepackage{graphicx,enumerate,stmaryrd,color,tikz}
\usepackage[all,2cell]{xy}
\xyoption{2cell}
\usetikzlibrary{patterns,snakes}

\newtheorem{theorem}{Theorem}
\newtheorem{lemma}[theorem]{Lemma}
\newtheorem{corollary}[theorem]{Corollary}
\newtheorem{proposition}[theorem]{Proposition}

\newtheorem{example}[theorem]{Example}

\usepackage[all]{xy}
\usepackage[active]{srcltx}
\usepackage[parfill]{parskip}
\usepackage{enumerate}

\font\sc=rsfs10

\newcommand{\cP}{\sc\mbox{P}\hspace{1.0pt}}

\newcommand{\Hom}{\operatorname{Hom}}

\def\hk#1{\textcolor[rgb]{0, 0, 0}{#1}}

\newcommand{\inv}{^{-1}}

\newcommand{\Ext}{\operatorname{Ext}}
\newcommand{\ext}{\operatorname{ext}}

\newcommand{\T}{\mathcal{T}_\bullet}
\newcommand{\exy}{{\mathrm{E}}_{y,x}}
\newcommand{\e}{\mathrm{E}}
\newcommand{\tE}{\widetilde{\mathrm{E}}}

\begin{document}
\title[Graded extensions of Verma modules]
{Graded extensions of Verma modules}

\author[H.~Ko and V.~Mazorchuk]
{Hankyung Ko and Volodymyr Mazorchuk}

\begin{abstract}
In this paper, we investigate extensions between graded Verma modules 
in the BGG category $\mathcal{O}$. In particular, we determine exactly 
which information about extensions between graded Verma modules is
given by the coefficients of the $R$-polynomials. We also give some
upper bounds for the dimensions of graded extensions between Verma 
modules in terms of Kazhdan-Lusztig combinatorics. We completely
determine all extensions between Verma module in the regular
block of category $\mathcal{O}$ for $\mathfrak{sl}_4$ and construct
various ``unexpected'' higher extensions between Verma modules. 
\end{abstract}

\maketitle

\section{Introduction and description of the results}\label{s1}

To determine extensions between Verma modules in 
Bernstein-Gelfand-Gelfand category $\mathcal{O}$ associated to a 
triangular decomposition of a semi-simple finite dimensional complex 
Lie algebra $\mathfrak{g}$ is a famous open problem. It was studied  in, for example, 
\cite{De,GJ,Ca1,Bo,Ma,Ab,Ca2}. The paper \cite{GJ} suggested a connection
between this problem and a combinatorial gadget, associated to the Weyl group 
of $\mathfrak{g}$, called the $R$-polynomials. Although not explicitly stated
in \cite{GJ}, the expectation that the dimensions of the extension groups
between Verma modules should be given by the coefficients of $R$-polynomials
became known as the {\em Gabber-Joseph conjecture}. Unfortunately, in \cite{Bo}
it was shown that this expectation is, in general, wrong. At the present stage, 
there is not even a conjectural answer to this problem. Some explicit results 
describing the extension groups between Verma modules in special cases can 
be found in \cite{Ca1,Ma,Ab,Ca2}. The main result of \cite{De} determines the
Euler characteristic (i.e., the alternating sum) for dimensions of such extensions.

In two recent papers \cite{KMM1,KMM2}, written jointly with Rafael Mr{\dj}en, we 
studied the dimensions of the first extension from a simple module to a Verma module.
In the case of the special linear Lie algebra, we gave an explicit formula for
this dimension, see \cite{KMM1}. Outside type $A$ the situation is more complicated.
However, an explicit formula can be given in many special cases, see \cite{KMM2}.
These results motivated us to take a new closer look at the classical problem
of extensions between Verma modules.

Our first main result, Theorem~\ref{thm3.2-1}, explicitly determines
the role that $R$-polynomials play in the theory of extensions between Verma modules.
Category $\mathcal{O}$ admits a natural $\mathbb{Z}$-graded lift. Delorme's formula
mentioned above turns out to have a natural graded analogue which asserts that 
the coefficients of $R$-polynomials determine the Euler characteristic for 
dimensions of certain graded extensions between Verma modules. This result
can be found in Section~\ref{s3}, see also Subsection~\ref{ss.r'poly}.

In Section~\ref{s4} we prove a number of general results about extension between
graded Verma modules. Combining this additional grading with the homological grading
gives a two-dimensional coordinate system (in $\mathbb{Z}^2$) in which the potential
region for non-zero graded extensions between two fixed Verma modules has the form
of a triangle
, see Figure~\ref{fig4.2-1.5}. If the distance between 
the indexes of Verma modules is small, then this triangle degenerates to its 
included side and all extensions are indeed described by the coefficients of 
the $R$-polynomials, see Corollary~\ref{cor4.4-3}. In fact, even in the general case,
there are specific situations where graded extensions between Verma modules are
given by the coefficients of  the $R$-polynomials, see Corollary~\ref{cor4.4-2}.
In Subsection~\ref{sskl} we give several general bounds for the dimension of a fixed 
extension between graded Verma modules in terms of the Kazhdan-Lusztig combinatorics.

Section~\ref{s5} discusses a few observations about extensions between Verma modules
related to the combinatorics of Bruhat order on the Weyl group. In particular, in
Proposition~\ref{prop5.2-1} we show that graded extension between Verma modules
are given by the coefficients of $R$-polynomials provided that the indexes of
the involved Verma modules are both boolean or are both coboolean.

In Section~\ref{s6} we present a number of concrete results in special cases.
In particular, in Theorem~\ref{thm6.3-1} we show that in the Weyl type $A_3$
(i.e., for the Lie algebra $\mathfrak{sl}_4$) all extensions between 
Verma modules are given by the coefficients of  the $R$-polynomials 
despite the nontrivial Kazhdan-Lusztig combinatorics.

Recall the triangle region of potential extensions between Verma modules 
mentioned above. It has one side included in the region and two sides that
are not included. We call the extensions corresponding to the included side
``expected'' and the extensions corresponding to the interior of the triangle
``additional''. If all extensions between two Verma modules are expected, then 
they are given by the coefficients of  the $R$-polynomials. In 
Subsection~\ref{s6.4}, we use the results from \cite{KMM1} to construct 
families of explicit non-zero additional first extensions between Verma modules
in type $A$. In Subsection~\ref{s6.5}, we use the results from Subsection~\ref{s6.4}
to construct families of explicit non-zero additional second 
extensions between Verma modules in type $A$.

The last section of the paper, Section~\ref{s7}, discusses the parabolic and singular cases
and the corresponding results similar to the ones that we obtain in the regular
case. In particular, Theorem~\ref{thm7.8} gives a necessary
condition for the graded extension algebra of Verma modules to be Koszul.
\vspace{5mm}

\subsection*{Acknowledgements} For the second author, the research is 
partially supported by the Swedish Research Council. Examples in Subsection~\ref{s6.5} were computed by SageMath.

\section{Preliminaries on category $\mathcal{O}$}\label{s2}

\subsection{Category $\mathcal{O}$}\label{s2.1}

Let $\mathfrak{g}$ be a semi-simple finite dimensional complex Lie algebra
with a fixed triangular decomposition 
$\mathfrak{g}=\mathfrak{n}_-\oplus \mathfrak{h}\oplus \mathfrak{n}_+$
(see \cite{Hu,MP} for details). Associated to this datum, we have
the Bernstein-Gelfand-Gelfand category $\mathcal{O}$, cf. \cite{BGG,Hu}.

Simple modules in $\mathcal{O}$ are exactly the simple highest weight modules
$L(\lambda)$, where $\lambda\in \mathfrak{h}^*$. For each such $\lambda$,
we also have the corresponding
\begin{itemize}
\item Verma module $\Delta(\lambda)$, 
\item dual Verma module $\nabla(\lambda)$,
\item indecomposable projective module $P(\lambda)$,
\item indecomposable injective module $I(\lambda)$,
\item indecomposable tilting module $T(\lambda)$.
\end{itemize}
The category $\mathcal O$ is a highest weight category with respect to the dominant order on $\mathfrak{h}^*$, where $\Delta(\lambda)$ are the standard modules and $\nabla(\lambda)$ are the costandard modules.

Consider the principal block $\mathcal{O}_0$ of $\mathcal{O}$, which is 
defined as the indecomposable direct summand contaning the trivial
$\mathfrak{g}$-module $L(0)$. Simple modules in $\mathcal{O}_0$ are
indexed by the elements of the Weyl group $W$ of $\mathfrak{g}$. 
For $w\in W$, we have the corresponding simple module $L_w:=L(w\cdot 0)$,
where $w\cdot 0$ denotes the usual dot-action of the Weyl group.
We similarly denote by $\Delta_w$, $\nabla_w$, $P_w$, $I_w$ and $T_w$
the other structural modules corresponding to the weight $w\in W$.
Then $\mathcal O_0$ is highest weight with respect to the opposite of the Bruhat order on $W$.

We use $\mathrm{Ext}$ and $\mathrm{Hom}$ to denote extensions and
homomorphisms in $\mathcal{O}$, respectively. The simple preserving duality
on $\mathcal{O}$ is denoted by $\star$.

\subsection{Graded category $\mathcal{O}$}\label{s2.2}

The category $\mathcal{O}_0$ admits a $\mathbb{Z}$-graded lift 
$\mathcal{O}_0^{\mathbb{Z}}$, see \cite{So}. All structural modules
in $\mathcal{O}_0$ admit graded lifts (unique up to isomorphism and 
shift of grading). We use the following notation for the standard graded lifts
of indecomposable structural modules:
\begin{itemize}
\item by ${L}_w$ the graded lift concentrated in degree $0$,
\item by ${\Delta}_w$ the graded lift with the top
in degree $0$,
\item by ${\nabla}_w$ the graded lift with the socle
in degree $0$,
\item by ${P}_w$ the graded lift  with the top
in degree $0$,
\item by ${I}_w$ the graded lift  with the socle
in degree $0$,
\item by ${T}_w$ the graded lift  having the unique
$L_w$ subquotient in degree $0$.
\end{itemize}

We denote by $\langle k\rangle$ the functor which shifts the grading,
with the convention that $\langle 1\rangle$ maps degree $0$ to degree $-1$.
We use $\mathrm{ext}$ and $\mathrm{hom}$ to denote extensions and
homomorphisms in $\mathcal{O}_0^{\mathbb{Z}}$, respectively. 
The graded version of $\star$ is also denoted by $\star$.

\hk{
\subsection{Bruhat order and the zeroth extensions}\label{ss.Vermainclusion}
We recall in this subsection the graded homomorphisms between Verma modules, which is well-knwon (see, for example, \cite[Chapter~7]{Di}). Let $\leq$ be the Bruhat order on $W$. Then we have
\begin{equation*}
\dim\hom(\Delta_x\langle a\rangle,\Delta_y\langle b\rangle ) = \begin{cases}
1 &\text{if $x\geq y$ and $a-b = \ell(x)-\ell(y)$}\\
0 &\text{otherwise}.
\end{cases}    
\end{equation*}
Moreover, any nonzero homomorphism between Verma modules is injective. 
Also, the nonzero homomorphisms $L_{w_0}=\Delta_{w_0}\to \Delta_y\langle \ell(y) - \ell(w_0) \rangle$, and their shifts, gives the socle of the Verma modules.}

\subsection{Combinatorics of category $\mathcal{O}_0^{\mathbb{Z}}$}\label{s2.3}

Let $\mathbf{H}$ denote the Hecke algebra of $W$ over $
\mathbb{Z}[v,v^{-1}]$
in the normalization of \cite{So2}. It has the standard basis
$\{H_w\,:\,w\in W\}$ and the Kazhdan-Lusztig (KL) basis $\{\underline{H}_w\,:\,w\in W\}$.
The KL polynomials $\{p_{x,y}\,:\,x,y\in W\}$ are the entires of 
the transformation matrix between these two bases, that is 
\begin{displaymath}
\underline{H}_y=\sum_{x\in W}p_{x,y}H_x, \text{ for all } y\in W.
\end{displaymath}
\hk{By construction, we have $p_{x,y}\in \mathbb Z[v]$ and $p_{x,y}=0$ for $x\not\leq y$. When $p_{x,y}\neq 0$, we have $\deg p_{x,y}=\ell(y)-\ell(x)$.}
For $x,y\in W$ and $k\in\mathbb{Z}$, we denote by $p_{x,y}^{(k)}$ the coefficient
at $v^k$ in $p_{x,y}$.

Taking the Grothendieck group gives rise to an isomorphism of 
$\mathbb{Z}[v,v^{-1}]$-modules as follows:
\begin{equation}\label{eq2.3-1}
\mathrm{Gr}(\mathcal{O}_0^{\mathbb{Z}}) \cong \mathbf{H},\qquad
[{\Delta}_w]\mapsto H_w, \text{ for } w\in W.
\end{equation}
Here the $\mathbb{Z}[v,v^{-1}]$-module structure on 
$\mathrm{Gr}(\mathcal{O}_0^{\mathbb{Z}})$ is given be letting
the element $v$ act as $\langle -1\rangle$.
\hk{By the Kazhdan-Lusztig theorem, see~\cite{KL,So}, this isomorphism maps ${P}_w$ to $\underline{H}_w$, for $w\in W$.
In particular, we have $p_{x,y}\in \mathbb Z_{\geq 0}[v]$, where its coefficients are the multiplicities of the (graded) filtration of $P_y$ by $\Delta_x$ (see Subsections~\ref{s2.1},~\ref{s2.2}).}

\subsection{Endofunctors of $\mathcal{O}_0$}\label{s2.5}

The category $\mathcal{O}_0$ is equipped with the action of various families
of functors, see \cite{BG,Ca1,AS,KM,MS,Hu} 
 and references therein. 
For $w\in W$, we denote by
\begin{itemize}
\item $\theta_w$ the indecomposable projective endofunctor of $\mathcal{O}_0$
sending $P_e$ to $P_w$, see \cite{BG};
\item $\mathrm{C}_w$ the corresponding shuffling functor, see \cite{Ca1,MS};
\item $\top_w$ the corresponding twisting functor, see \cite{AS,KM};
\end{itemize}

The action of the monoidal category $\cP^\mathbb{Z}$ of graded projective
functors on $\mathcal{O}_0^{\mathbb{Z}}$ categorifies the right regular
$\mathbf{H}$-module.

\subsection{$R$-polynomials}\label{s2.4}

The isomorphism in \eqref{eq2.3-1} equips the algebra $\mathbf{H}$ with the following
$\mathbb{Z}[v,v^{-1}]$-bases:
\begin{itemize}
\item $\{H_w=[\Delta_w]\,:\,w\in W\}$, the standard basis;
\item $\{\underline{H}_w=[P_w]\,:\,w\in W\}$, the KL basis;
\item  $\{[\nabla_w]\,:\,w\in W\}$, the costandard basis;
\item  $\{[L_w]\,:\,w\in W\}$, the dual KL basis;
\item  $\{[I_w]\,:\,w\in W\}$;
\item  $\{[T_w]\,:\,w\in W\}$.
\end{itemize}
The {\em $R$-polynomials} $\{r_{x,y}\,:\,x,y\in W\}$ are defined as the entries of
the transformation matrix between the standard and the costandard bases, i.e.:
\begin{displaymath}
[\Delta_y]=\sum_{x\in W}r_{x,y}[\nabla_x], \text{ for all } y\in W.
\end{displaymath}
Note that $r_{x,y}\in\mathbb{Z}[v,v^{-1}]$, by definition. 
For $x,y\in W$ and $k\in\mathbb{Z}$, we denote by $r_{x,y}^{(k)}$ the coefficient
at $v^k$ in $r_{x,y}$.

As $\Delta_{w_0}=\nabla_{w_0}$, we have
\begin{equation}\label{eq2.4-1}
r_{x,w_0}=
\begin{cases}
1,& x=w_0;\\
0,& \text{otherwise}.
\end{cases}
\end{equation}

For $w\in W$ and $s\in S$ such that $ws>w$, we have
\begin{displaymath}
[\theta_s\Delta_w]=v[\theta_s\Delta_{ws}]=[\Delta_{ws}]+v[\Delta_w]\,\text{ and }\,
[\theta_s\nabla_w]=v^{-1}[\theta_s\nabla_{ws}]=[\nabla_{ws}]+v^{-1}[\nabla_w].
\end{displaymath}
From this, we have the following recursive formula for $R$-polynomials:
For $x,y\in W$ and $s\in S$ such that $ys<y$, we have:
\begin{equation}\label{eq2.4-2}
r_{x,ys}=
\begin{cases}
r_{xs,y},& xs<x;\\ 
r_{xs,y}+ (v^{-1}-v)r_{x,y},& xs>x.
\end{cases}
\end{equation}
Together, Formulae~\eqref{eq2.4-1} and \eqref{eq2.4-2} determine the
family of $R$-polynomials uniquely.

Please note that our indexing of $R$-polynomials differs from
the usual one in \cite{KL,BB} by a $w_0$-shift.
For more information on $R$-polynomials, we refer to \cite{KL}
and \cite[Section~5.3]{BB}.

\subsection{Koszul and Koszul-Ringel dualities}\label{s2.6}

Consider $\mathcal{D}^b(\mathcal{O}_0^{\mathbb{Z}})$ the bounded derived category of $\mathcal{O}_0^{\mathbb{Z}}$. The category $\mathcal{D}^b(\mathcal{O}_0^{\mathbb{Z}})$ has a famous
auto-equivalence called the {\em Koszul duality}
and denoted by $\mathrm{K}$, see \cite{So,BGS,MOS}. 
It has the following properties:
\begin{itemize}
\item $\mathrm{K}$ sends $P_w$ to $L_{w^{-1}w_0}$;
\item $\mathrm{K}$ sends $\Delta_w$ to $\nabla_{w^{-1}w_0}$;
\item $\mathrm{K}$ sends $L_w$ to $I_{w^{-1}w_0}$;
\item $\mathrm{K}\langle j\rangle\cong [j]\langle -j\rangle\mathrm{K}$.
\end{itemize}

Another famous auto-equivalence of $\mathcal{D}^b(\mathcal{O}_0^{\mathbb{Z}})$
is its Ringel self-duality given by the derived twisting functor
$\mathcal{L}\top_{w_0}$, see \cite{So1-2}. It has the following properties:
\begin{itemize}
\item $\mathcal{L}\top_{w_0}$ sends $P_w$ to $T_{w_0w}$;
\item $\mathcal{L}\top_{w_0}$ sends $T_{w}$ to $I_{w_0w}$;
\item $\mathcal{L}\top_{w_0}$ sends $\Delta_{w}$ to $\nabla_{w_0w}$.
\end{itemize}

\hk{
The composition $\mathrm{RK}:=(\mathcal L\top_{w_0})^{-1}\circ\mathrm{K}$ of the Koszul and Ringel self-dualities results in
the {\em Koszul-Ringel self-duality}. We have $\mathrm{RK}(\Delta_w)=\Delta_{w_0w^{-1}w_0}$, and thus $\mathrm{RK}^2(\Delta_w)=\Delta_w$. Since $\{\Delta_w\}_{w\in W}$ generates $\mathcal{D}^b(\mathcal{O}_0^{\mathbb{Z}})$, it follows that the Koszul-Ringel self-duality is an involution.
It has the following properties:
\begin{itemize}
\item $\mathrm{RK}$ sends $T_w$ to $L_{w_0w^{-1}w_0}$;
\item $\mathrm{RK}$ sends $L_w$ to $T_{w_0w^{-1}w_0}$;
\item $\mathrm{RK}$ sends $\Delta_w$ to $\Delta_{w_0w^{-1}w_0}$;
\item $\mathrm{RK}$ sends $\nabla_w$ to $\nabla_{w_0w^{-1}w_0}$.
\end{itemize}}

We refer to \cite{Ma2} for further details. 

A concrete realization of this self-duality is via the category $\mathscr{LC}(T)$ 
of linear complexes of tilting objects in $\mathcal{O}_0^{\mathbb{Z}}$. 
Recall that a complex $\mathcal{T}_\bullet$ of tilting objects in called
{\em linear} provided that each summand of each $\mathcal{T}_i$ has the form
$T_w\langle i\rangle$, for some $w\in W$. The essence of the Koszul-Ringel self-duality
is that $\mathrm{RK}$ restricts to an equivalence between the categories $\mathscr{LC}(T)$ and $\mathcal{O}_0^{\mathbb{Z}}$ where
\begin{itemize}
\item the tilting module $T_w$ (considered as a complex) is sent to 
the simple module $L_{w_0w^{-1}w_0}$;
\item the (linear!) complex of tilting modules representing $L_w$ is sent to the tilting module $T_{w_0w^{-1}w_0}$;
\item the linear tilting coresolution of  $\Delta_w$ is sent to 
the module $\Delta_{w_0w^{-1}w_0}$;
\item the linear tilting resolution of  $\nabla_w$ is sent to the module 
$\nabla_{w_0w^{-1}w_0}$.
\end{itemize}
We use the notation $\mathcal{T}_\bullet(L_w)$, $\mathcal{T}_\bullet(\Delta_w)$
and $\mathcal{T}_\bullet(\nabla_w)$ for the linear complexes of tilting modules that 
represent $L_w$, $\Delta_w$ and $\nabla_w$, for $w\in W$, respectively.

\section{Delorme formulae}\label{s3}

\subsection{Ungraded Delorme formula}\label{s3.1}

The following results is proved in \cite{De}.

\begin{proposition}\label{prop3.1-1}
For $x,y\in W$, we have
\begin{displaymath}
\sum_{i\geq 0}(-1)^i\dim\mathrm{Ext}^i_{\mathcal{O}}(\Delta_x,\Delta_y)=\delta_{x,y}. 
\end{displaymath}
\end{proposition}

\begin{proof}
First, we claim that, for any $M\in\mathcal{D}^b(\mathcal{O}_0)$, we have the following
relation in $\mathrm{Gr}(\mathcal{D}^b(\mathcal{O}_0))$:
\begin{equation}\label{eq3.1-2}
[M]= \sum_{x\in W}\sum_{i\geq 0}(-1)^i\dim\mathrm{Ext}^i_{\mathcal{O}}(\Delta_x,M)[\Delta_x].
\end{equation}
Indeed, for $M=\nabla_y$, this formula follows by combining the fact that standard
and costandard modules in $\mathcal{O}_0$ are homologically orthogonal with the fact that 
$[\Delta_x]=[\Delta_x^\star]=[\nabla_x]$, since $\star$ is simple preserving.
For the general $M$, Formula~\eqref{eq3.1-2} now follows using the additivity of both sides
with respect to distinguished triangles and the fact that costandard modules generate
$\mathcal{D}^b(\mathcal{O}_0)$.
 
The claim of the proposition is obtaned from Formula~\eqref{eq3.1-2} by plugging in
$M=\Delta_y$ and using that $\{[\Delta_x]\,:\,x\in W\}$ is a basis in
$\mathrm{Gr}(\mathcal{D}^b(\mathcal{O}_0))$.
\end{proof}

\subsection{Graded Delorme formula}\label{s3.2}

The following is a natural graded lift of Proposition~\ref{prop3.1-1}.
This statement explicitly explains the role which $R$-polynomials play
in the theory of extensions between Verma modules.

\begin{theorem}\label{thm3.2-1}
For $x,y\in W$ and $k\in\mathbb{Z}$, we have
\begin{displaymath}
\sum_{i\geq 0}(-1)^i\dim\mathrm{ext}^i(\Delta_x\langle k\rangle,\Delta_y)=r_{x,y}^{(k)}. 
\end{displaymath}
\end{theorem}

\begin{proof}
First, we claim that, for any $M\in\mathcal{D}^b(\mathcal{O}_0^{\mathbb{Z}})$, 
we have the following relation in $\mathrm{Gr}(\mathcal{D}^b(\mathcal{O}_0^{\mathbb{Z}}))$:
\begin{equation}\label{eq3.2-2}
[M]= \sum_{x\in W}\sum_{i\geq 0}\sum_{k\in\mathbb{Z}}(-1)^i
\dim\mathrm{ext}^i(\Delta_x\langle k\rangle,M)[\nabla_x\langle k\rangle].
\end{equation}
Indeed, for $M=\nabla_y\langle m\rangle$, this formula follows from the fact that standard
and costandard modules in $\mathcal{O}_0^{\mathbb{Z}}$ are homologically orthogonal.
For the general $M$, Formula~\eqref{eq3.2-2} now follows using the additivity of both sides
with respect to distinguished triangles and the fact that costandard modules generate
$\mathcal{D}^b(\mathcal{O}_0^{\mathbb{Z}})$.
 
The claim of the proposition is obtaned from Formula~\eqref{eq3.2-2} by plugging in
$M=\Delta_y$ and using that $\{[\Delta_x\langle k\rangle]\,:\,x\in W,k\in\mathbb{Z}\}$ 
is a basis in $\mathrm{Gr}(\mathcal{D}^b(\mathcal{O}_0^{\mathbb{Z}}))$ and the definition of
$R$-polynomials.
\end{proof}

\section{General results}\label{s4}

\subsection{General  setup}\label{s4.1}

For $i,j\in\mathbb{Z}$, $x,y\in W$ and $k\in\mathbb{Z}_{\geq 0}$, set
\begin{displaymath}
E(x,y,i,j,k):=\dim\mathrm{ext}^k(\Delta_x\langle i\rangle,\Delta_y\langle j\rangle). 
\end{displaymath}
An ultimate goal would be to find a formula for $E(x,y,i,j,k)$. 
Let us start by listing some straightforward properties:

\begin{proposition}\label{prop4.1-1}
For $i,j\in\mathbb{Z}$, $x,y\in W$ and $k\in\mathbb{Z}_{\geq 0}$, we have:
\begin{enumerate}[$($a$)$]
\item \label{prop4.1-1.1} $E(x,y,i,j,k)=E(x,y,i+a,j+a,k)$, for all $a\in\mathbb{Z}$.
\item \label{prop4.1-1.2} $E(x,y,i,j,k)\neq 0$ implies  $x\geq y$.
\item \label{prop4.1-1.3} $E(x,y,i,j,k)=E(w_0x^{-1}w_0,w_0y^{-1}w_0,-i,-j,k+j-i)$.
\item \label{prop4.1-1.4} $E(x,y,i,j,k)=E(x^{-1},y^{-1},-i,-j,k+j-i)$.
\item \label{prop4.1-1.5} $E(x,y,i,j,k)=E(w_0y,w_0x,-i,-j,k)$.
\end{enumerate}
\end{proposition}

\begin{proof}
Property~\eqref{prop4.1-1.1} follows from the fact that the shift of grading is 
an auto-equivalence. Property~\eqref{prop4.1-1.2} is a ususal property of standard 
modules in highest weight categories. Property~\eqref{prop4.1-1.3} follows from
Koszul-Ringel duality. Property~\eqref{prop4.1-1.4} follows from 
Property~\eqref{prop4.1-1.3}  since conjugation by $w_0$ corresponds to an 
automorphism of the Dynkin diagram, which induces a(n highest weight) auto-equivalence on $\mathcal O_0$. Property~\eqref{prop4.1-1.5} follows by
applying first $\mathcal{L}\top_{w_0}$ and then $\star$.
\end{proof}

Due to Proposition~\ref{prop4.1-1}\eqref{prop4.1-1.1}, we can consider the case $j=0$.
We define the polynomial $\tE_{y,x}(\upsilon,\omega)$ as follows:
\begin{displaymath}
\tE_{y,x}(\upsilon,\omega)= 
\sum_{i\in\mathbb{Z}}\sum_{k\in\mathbb{Z}_{\geq 0}} E(x,y,i,0,k)\upsilon^{-i}\omega^k.
\end{displaymath}
Note that $\tE_{y,x}$ is polynomial in $\omega$ and a Laurent polynomial
in $\upsilon$. \hk{Theorem~\ref{thm3.2-1} says $\tE_{x,y}(\upsilon,-1)=r_{x,y}(\upsilon)$.}

\hk{A convenient normalization is via the change of variables $u:=\upsilon^{-1}\omega, v:=\omega$. We thus obtain $\exy\in\mathbb Z[u^{\pm 1},v^{\pm 1}]$ such that
\begin{equation*}
    \exy(u,v) = \sum_{i,k\in\mathbb Z_{\geq 0}}E(x,y,k-i,0,k)u^iv^k.
\end{equation*}
Here we have $\exy(-v,-1)=r_{x,y}(v)$, and Proposition~\ref{prop4.1-1} is expressed as follows. 
\begin{proposition}\label{prop4poly}
We have
\begin{enumerate}[$($a$)$]
\item \label{prop5.2} $\exy = 0$ unless $y\leq x$.
\item \label{prop5.3} $\exy(u,v) = \e_{w_0y^{-1}w_0,w_0x^{-1}w_0}(v,u)$.
\item \label{prop5.4} $\exy(u,v) =\e_{y^{-1},x^{-1}}(v,u)$.
\item \label{prop5.5} $\exy(u,v)=\e_{w_0x,w_0y}(uv^{-1},v)$.
\end{enumerate}
\end{proposition}
}

\subsection{Bounds in terms of KL polynomials}\label{sskl}

\begin{proposition}\label{klkl}
For $x,y\in W$ we have
    \begin{equation}\label{dhom=pp}
    \sum_{a,b\in\mathbb Z} \dim\hom(\Delta_y\langle b-a \rangle,\mathcal T_a(\Delta_x))u^b v^{a} = \sum_{z\in W} p_{y w_0,z w_0}(u) p_{x,z}(v).
\end{equation}
   (Note that the summand $ p_{y w_0,z w_0}(u) p_{x,z}(v)$ is zero unless $x\leq z\leq y$.)
\end{proposition}
\begin{proof}
The Koszul-Ringel duality gives
\begin{equation}
    [\Delta_{w_0x\inv w_0}:L_{w_0z\inv w_0}\langle -a\rangle] = [\T(\Delta_x):T_z\langle a \rangle[a] ].
\end{equation}
The left hand side is, by the BGG reciprocity, equal to
$[P_{w_0z\inv w_0}:\Delta_{w_0x\inv w_0}\langle -a\rangle]$. 
It follows that
\begin{equation}\label{TinT}
    \sum_a[\T(\Delta_x):T_z\langle a \rangle[a] ]v^a = p_{w_0x\inv w_0,w_0z\inv w_0}(v) = p_{x,z}(v). 
\end{equation}

On the other hand, for each $z\in W$, we have
\[\dim \hom(\Delta_y,T_z\langle b\rangle) = [T_z : \nabla_y\langle -b\rangle]  = [P_{z w_0}:\Delta_{y w_0}\langle -b\rangle],\]
and thus
\begin{equation}\label{DinT}
    \sum_b \dim \hom(\Delta_y,T_z\langle b\rangle)u^b = p_{y w_0,z w_0}(u). 
\end{equation}

Combining \eqref{DinT} and \eqref{TinT}, we obtain the claimed equation.    
\end{proof}

Given $p,q\in \mathbb Z[u^\pm,v^\pm]$, we write $p\leq q$ if the coefficients of each monomial is smaller for $p$ than for $q$, that is, $p_{ij}\leq q_{ij}$ for all $i,j\in \mathbb Z$ where $p=\sum_{ij}p_{ij}u^iv^j,q=\sum_{ij}q_{ij}u^iv^j$.
\begin{corollary}\label{cor.klbound}
    For $x,y\in W$ we have
    \begin{equation}
        \e_{x,y}(u,v)\leq \sum_{z\in W} p_{y w_0,z w_0}(u) p_{x,z}(v).
    \end{equation}
\end{corollary}
    \begin{proof}
Since $\operatorname{ext}^a(\Delta_y\langle b-a\rangle,\Delta_x)$ is computed as the homology of $\hom(\Delta_y\langle b-a \rangle [a],\T(\Delta_x))$, the claim follows from  Proposition~\ref{klkl}.
\end{proof}

\begin{corollary}\label{cor.shiftboundtilde}
    We have $\exy\in \mathbb Z[u,v]$. Moreover, for $x\geq y$ we have
    \begin{enumerate}
        \item $\deg_u \exy = \ell(x)-\ell(y)$;
        \item $\deg_v \exy = \ell(x)-\ell(y)$;
        \item $\deg \exy = \ell(x)-\ell(y)$;
        \item\label{parity} the degree of each monomial appearing in $\exy$ has the same parity as $\ell(x)-\ell(y)$,
    \end{enumerate}
    where $\deg_u,\deg_v$ denote the degrees with respect to the variables $u,v$, respectively. In fact, the coefficients of $\exy(u,v)$ at $u^{\ell(x)-\ell(y)}$ and at $v^{\ell(x)-\ell(y)}$ are both $1$.
\end{corollary}
\begin{proof}
     Since $p_{w,w'}(\upsilon)\in\mathbb Z_{\geq 0}[\upsilon]$ with the parity vanishing property and $\deg p_{w,w'} = \ell(w')-\ell(w)$ for all $w,w'\in W$ with $w\leq w'$ (see Subsection~\ref{s2.3}), Corollary~\ref{cor.klbound} provides the first statement, the inequality ``$\leq$'' in all three numbered claims, and \eqref{parity}. To have ``$=$'' in the numbered claims, it is enough to prove the last remark. But the coefficients of $\exy(u,v)$ at $u^{\ell(x)-\ell(y)}$ and at $v^{\ell(x)-\ell(y)}$ are the same by Proposition~\ref{prop4.1-1}\eqref{prop4.1-1.3}, where the former is the dimension of $\hom(\Delta_x\langle \ell(x)-\ell(y)\rangle,\Delta_y)$. The latter space consists of the unique inclusion between Verma modules (see Subsection~\ref{ss.Vermainclusion}), and thus has dimension one. This completes the proof.
\end{proof}

\begin{corollary}\label{cor.shiftbound}
    If $E(x,y,i,0,k)\neq 0$ for $x\geq y$, then
\begin{enumerate}[$($a$)$]
\item\label{prop4.2-1.1} $0\leq k\leq \ell(x)-\ell(y)$;
\item\label{prop4.2-1.2} $-2k+\ell(y)-\ell(x)\leq i\leq -k$;
\item\label{prop4.2-1.3} if $k=0$, then $i=\ell(y)-\ell(x)$;
\item\label{prop4.2-1.4} if $k=-i$, then $k=-i=\ell(x)-\ell(y)$;
\item\label{prop4.2-1.5} $\ell(y)-\ell(x)-i$ is even.
\end{enumerate}
Thus, each $(k,i)$ with nonzero $E(x,y,i,0,k)$ is in the violet region in
 Figure~\ref{fig4.2-1.5}.
\end{corollary}
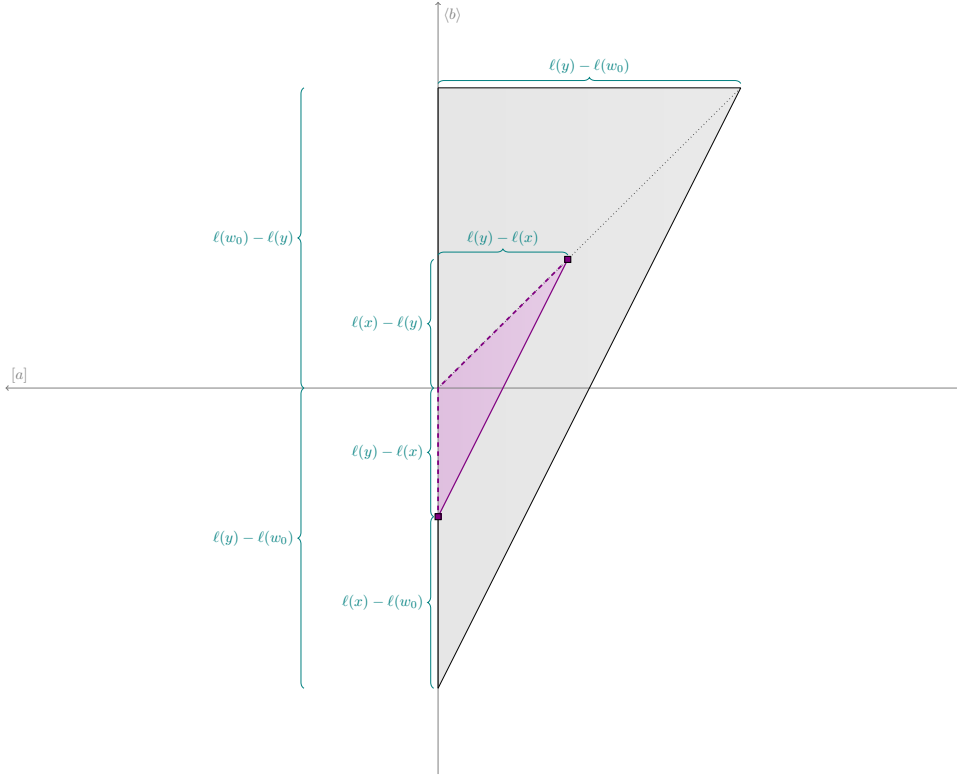
\begin{figure}
\resizebox{\textwidth}{!}{
\begin{tikzpicture}
\shade[left color = gray!20,right color = gray!17] (0,-7) -- (0,7) -- (7,7) -- cycle;
\shade[left color = violet!25,right color = violet!20] (0,-3) -- (0,0) -- (3,3) -- cycle;
\draw[gray, thin,  ->] (12,0) -- (-10,0) node[anchor=south west] {$\left[ a\right]$};
\draw[gray, thin,  ->] (0,-9) -- (0,9) node[anchor=north west] {$\langle b\rangle$};
\draw[black, thin,  -] (0,-7) -- (7,7);
\draw[black, thin,  -] (0,7) -- (7,7);
\draw[black, thin,  -] (0,7) -- (0,0);
\draw[black, thin,  -] (0,-3) -- (0,-7);
\draw[black, dotted, thin,  -] (0,0) -- (7,7);
\draw[violet, thick,  -] (0,-3) -- (3,3);
\draw[violet, dashed, very thick,  -] (0,0) -- (0,-3);
\draw[violet, dashed, very thick,  -] (3,3) -- (0,0);
\draw [fill=violet] (-0.07,-3.07) rectangle +(0.14,0.14);
\draw [fill=violet] (2.93,2.93) rectangle +(0.14,0.14);
\draw[teal, thin, decoration={brace,amplitude=4pt,raise=1mm},decorate]
(0,-3) -- (0,0) node [midway,left,xshift=-1mm] {\color{teal}$\ell(y)-\ell(x)\,\,$};
\draw[teal, thin, decoration={brace,amplitude=4pt,raise=1mm},decorate]
(0,-7) -- (0,-3) node [midway,left,xshift=-1mm] {\color{teal}$\ell(x)-\ell(w_0)\,\,$};
\draw[teal, thin, decoration={brace,amplitude=4pt,raise=1mm},decorate]
(-3,-7) -- (-3,0) node [midway,left,xshift=-1mm] {\color{teal}$\ell(y)-\ell(w_0)\,\,$};
\draw[teal, thin, decoration={brace,amplitude=4pt,raise=1mm},decorate]
(0,0) -- (0,3) node [midway,left,xshift=-1mm] {\color{teal}$\ell(x)-\ell(y)\,\,$};
\draw[teal, thin, decoration={brace,amplitude=4pt,raise=1mm},decorate]
(-3,0) -- (-3,7) node [midway,left,xshift=-1mm] {\color{teal}$\ell(w_0)-\ell(y)\,\,$};
\draw[teal, thin, decoration={brace,amplitude=4pt,raise=1mm},decorate]
(0,3) -- (3,3) node [midway,yshift=5mm] {\color{teal}$\ell(y)-\ell(x)$}; 
\draw[teal, thin, decoration={brace,amplitude=4pt,raise=1mm},decorate]
(0,7) -- (7,7) node [midway,yshift=5mm] {\color{teal}$\ell(y)-\ell(w_0)$}; 
\end{tikzpicture}
}
\caption{The homomorphisms from $\Delta_x$ to $\T(\Delta_y)$, and thus the extensions from $\Delta_x$ to $\Delta_y$, are in the violet region; the composition factors in $\T(\Delta_y)$ are in the grey region.}\label{fig4.2-1.5}
\end{figure}

\begin{proof}
The claims are exactly the claims in Corollary~\ref{cor.shiftboundtilde} via the change of variables.
\end{proof}

The bound given in Corollary~\ref{cor.klbound} does not take into account the differentials in the complex of homomorphisms between $\Delta_y$ and $\T(\Delta_x))$ and can be lowered in various ways. 
We record one such strenghtening of the bound.

\begin{proposition}\label{prop4.6-2}
Let $x,y\in W$ be such that $x\geq y$. Then, for $k\in\mathbb{Z}_{\geq 0}$ 
and $i\in\mathbb{Z}$ such that $2k-i\neq \ell(x)-\ell(y)$, we have:
\[E(x,y,i,0,k)\leq 
\sum_{w\in W}p^{(k)}_{w_0y^{-1}w_0,w_0w^{-1}w_0}
p_{w_0x,w_0w}^{(k-i)}-
\max_{\ell(w)=\ell(y)+k, w\geq y}p_{w_0x,w_0w}^{(k-i)}.
\]
\end{proposition}

\begin{proof}
Let $w\in W$ such that $w\geq y$ and $\ell(w)=\ell(y)+k$.
Via Koszul-Ringel duality, the inclusion $\Delta_{w_0w^{-1}w_0}\langle k\rangle\hookrightarrow
\Delta_{w_0y^{-1}w_0}$
gives an injective (component-wise) homomorphism of complexes 
$\mathcal{T}_\bullet(\Delta_w)\langle -k\rangle[k]\hookrightarrow 
\mathcal{T}_\bullet(\Delta_y)$.

Since each hom space between Verma modules
is concentrated in one degree, 
any nonzero homomorphism, say $\phi:\Delta_x\langle i \rangle \to T_w=\mathcal T_0(\Delta_w)$ does not give rise to a homomorphism of complexes, that is,
$d \circ \phi\neq 0$, where $d$ is (the restriction of) the differential in $\mathcal{T}_\bullet(\Delta_w)$.

Since $\mathcal{T}_\bullet(\Delta_w)\langle -k\rangle[k]$ is a subcomplex of 
$\mathcal{T}_\bullet(\Delta_y)$, we still have $d\circ \phi\neq 0$ when $d$ is the differential in 
$\mathcal{T}_\bullet(\Delta_y)$. Hence $\phi$ does not 
contribute to an (appropriately shifted) extension from $\Delta_x$
to $\Delta_y$. The claim follows.
\end{proof}

\subsection{The exact information given by the $R$-polynomials}\label{ss.r'poly}
Theorem~\ref{thm3.2-1} implies that the coefficient $r_{x,y}^{(i)}$ of the
$R$-polynomial $r_{x,y}$ has the following interpretation in terms of 
the violet triangle in Figure~\ref{fig4.2-1.5}. We need to consider the 
integral points in the intersection of the line $b=i$ with the violet triangle
(the dashed parts excluded) as given by the small black boxes in the
following picture:

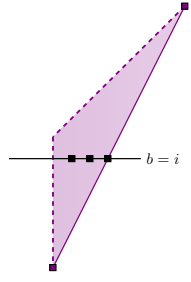
\begin{center}
\resizebox{2.5cm}{!}{
\begin{tikzpicture}
\shade[left color = violet!25,right color = violet!20] (0,-3) -- (0,0) -- (3,3) -- cycle;
\draw[violet, thick,  -] (0,-3) -- (3,3);
\draw[violet, dashed, very thick,  -] (0,0) -- (0,-3);
\draw[violet, dashed, very thick,  -] (3,3) -- (0,0);
\draw [fill=violet] (-0.07,-3.07) rectangle +(0.14,0.14);
\draw [fill=violet] (2.93,2.93) rectangle +(0.14,0.14);
\draw[ thick,  -] (-1,-0.5) -- (2,-0.5) node[anchor=west] {$b=i$};
\draw [fill=black] (1.18,-0.57) rectangle +(0.14,0.14);
\draw [fill=black] (0.77,-0.57) rectangle +(0.14,0.14);
\draw [fill=black] (0.36,-0.57) rectangle +(0.14,0.14);
\end{tikzpicture}
}
\end{center}

The coefficient $r_{x,y}^{(i)}$ is exactly the alternating sum of the
dimensions of those extensions from (a shift of) $\Delta_x$ to $\Delta_y$,
where the top of $\Delta_x$ is shifted by the coordinates of these small black boxes.
Note that the dashed sides are excluded, except for the south and the east vertices
of the triangle.

\subsection{Consequences}\label{s4.4}

\begin{corollary}\label{cor4.4-1}
Let $\ell(x)-\ell(y)\equiv i\mod 2$.
If the intersection of the line $b=i$ with the violet triangle
contains exactly one integral point, say $(-a,b)$, then
\begin{equation}\label{eq4.4-5}
a=\frac{i+\ell(x)-\ell(y)}{2}\quad\text{ and }\quad
\dim\mathrm{ext}^{a}(\Delta_x\langle i\rangle,\Delta_y)=|r_{x,y}^{(i)}|. 
\end{equation}
\end{corollary}

\begin{proof}
This follows directly from the discussion in Subsection~\ref{ss.r'poly}. 
\end{proof}

\begin{corollary}\label{cor4.4-2}
Let $x,y\in W$ with $x\geq y$. Then Formula~\eqref{eq4.4-5} holds 
for the following values of $i$:
\begin{displaymath}
i\in\{\ell(x)-\ell(y),\ell(x)-\ell(y)-2,2-\ell(x)+\ell(y),\ell(y)-\ell(x)\}. 
\end{displaymath}
\end{corollary}

\begin{proof}
If $i=\ell(x)-\ell(y)$, then the intersection of $b=i$ with the violet
triangle consists of the east vertex of the triangle.
If $i=\ell(y)-\ell(x)$, then the intersection of $b=i$ with the violet
triangle consists of the south vertex of the triangle.
If $i=\ell(x)-\ell(y)-2$, then the intersection of $b=i$ with the violet
triangle consists of the vertcies $(\ell(y)-\ell(x)+1,\ell(x)-\ell(y)-2)$
and $(\ell(y)-\ell(x)+2,\ell(x)-\ell(y)-2)$, however, the latter one belongs
to the dashed line. Similarly, if $i=\ell(y)-\ell(x)+2$, then the 
intersection of $b=i$ with the violet triangle consists of the vertcies 
$(-1,\ell(y)-\ell(x)+2)$
and $(0,\ell(y)-\ell(x)+2)$, however, the latter one belongs
to the dashed line.

This means that, in all four cases, we have exactly one relevant integral point.
Now the claim follows from Corollary~\ref{cor4.4-1}.
\end{proof}

\begin{corollary}\label{cor4.4-3}
Let $x,y\in W$ with $x\geq y$. If $\ell(x)-\ell(y)\leq 3$, then
Formula~\eqref{eq4.4-5} holds.
\end{corollary}

\begin{proof}
The claim of the corollary 
follows directly from Corollary~\ref{cor4.4-2} since, under the assumption
$\ell(x)-\ell(y)\leq 3$, the values of $i$ listed in Corollary~\ref{cor4.4-2}
cover all possibilities for potentially non-zero extensions.
\end{proof}

\subsection{Expected vs additional extensions}\label{s4.5}

As illustrated in Figure~\ref{fig4.2-1.5}, \hk{the proof of Proposition~\ref{klkl}} 
imply that the non-zero $E(x,y,i,0,k)$ split naturally
into two types:
\begin{itemize}
\item The cases when $k=\frac{i+\ell(x)-\ell(y)}{2}$, i.e., the top of 
$\Delta_x\langle i\rangle$ lies on the solid violet side of the violet triangle.
We call such cases {\em expected}.
\item All other cases. In these cases the top of $\Delta_x\langle i\rangle$ 
belongs to the interior of the violet triangle. We call such cases {\em additional}. 
\end{itemize}

\hk{Then the discussion in Subsection~\ref{ss.r'poly} gives the following statement.}
\begin{corollary}\label{cor4.5-3}
If $x,y\in W$ and $i\in\mathbb{Z}$ are such that all extensions 
between $\Delta_x\langle i\rangle$  and $\Delta_y$ are known to
be expected, then  Formula~\eqref{eq4.4-5} holds. 
\end{corollary}

Each non-zero element in $\mathrm{ext}^k(\Delta_x\langle i\rangle,\Delta_y)$ is realized
via a non-zero homomorphism from $\Delta_x\langle i\rangle$ to 
$\mathcal{T}_{k}(\Delta_y)$. Since the latter has a standard filtration and all
non-zero morphisms between standard modules are injective, the map from 
$\Delta_x\langle i\rangle$ to $\mathcal{T}_{k}(\Delta_y)$ must be injective,
in particular, it must be injective when restricted to the simple socle of
$\Delta_x\langle i\rangle$.

For expected extensions, the socle of $\Delta_x\langle i\rangle$ is on the diagonal side
of the gray triangle. This means that the homomorphism from $\Delta_x\langle i\rangle$
to $\mathcal{T}_{k}(\Delta_y)$ which realizes this extension has image inside 
the direct sum of all $T_w\langle k\rangle$, where the sum is taken over all
$w$ such that $w\geq y$ and $\ell(w)=\ell(y)+k$.

For additional extensions, the socle of $\Delta_x\langle i\rangle$ is in the interior
of the gray triangle. This means that the homomorphism from $\Delta_x\langle i\rangle$
to $\mathcal{T}_{k}(\Delta_y)$ which realizes this extension 
\begin{itemize}
\item either induces a non-zero map to a summand of  $\mathcal{T}_{k}(\Delta_y)$
different from the $T_w\langle k\rangle$ as in the previous paragraph;
\item or induces a non-zero map to some $T_w\langle k\rangle$ as in the previous paragraph,
in which case the socle of this $T_w\langle k\rangle$ is not simple and this induced
map maps the socle of $\Delta_x\langle i\rangle$ to the socle part of 
$T_w\langle k\rangle$ which lives in a non-maximal degree.
\end{itemize}
Each of thess cases is only possible in the situation when some KL-polynomials
are nontrivial. 

The first situation is possible only if $\mathcal{T}_{k}(\Delta_y)$ contains 
a summand $T_w\langle k\rangle$, for some $w\geq y$ which violates $\ell(w)=\ell(y)+k$.
Via the Koszul-Ringel duality, the fact that $T_w\langle k\rangle$ is a sumand
of $\mathcal{T}_{k}(\Delta_y)$ means that
$[\Delta_{w_0y^{-1}w_0}:L_{w_0w^{-1}w_0}\langle -k\rangle]\neq 0$. The latter multiplicity
is exactly the coefficient at $v^k$ in the KL polynomial $p_{w_0y^{-1}w_0,w_0w^{-1}w_0}$.
Since $\ell(w)\neq \ell(y)+k$, this means that $p_{w_0y^{-1}w_0,w_0w^{-1}w_0}$ is not trivial.

The second situation is possible only if the socle of some $T_w\langle k\rangle$, where
$w\geq y$ and $\ell(w)=\ell(y)+k$, is not simple. \hk{Since the socle is a direct sum of $L_{w_0}$, the latter is if and only if $(T_w:\Delta_{w_0})>1$, which is equivalent, via Koszul duality, to the polynomial $p_{e,w_0w}$ being nontrivial.}

An immediate corollary of this discussion is the following:

\begin{corollary}\label{cor4.5-1}
Formula~\eqref{eq4.4-5} holds in all cases when $W$ has rank $2$.
\end{corollary}

\begin{proof}
In rank $2$ case, all KL polynomials are trivial. Therefore the only non-zero
extensions between Verma modules are those where the top of $\Delta_x\langle i\rangle$
is on the solid violet side of the violet triangle. Therefore Formula~\eqref{eq4.4-5}
follows from Theorem~\ref{thm3.2-1}.
\end{proof}

More generally, the same argument gives:

\begin{corollary}\label{cor4.5-2}
Let $y\in W$ be such that $p_{w_0y^{-1}w_0,w_0w^{-1}w_0}$
and $p_{e,w_0w}$ are trivial, for all $w\geq y$. Then 
Formula~\eqref{eq4.4-5} holds for all $x\in W$ such that $x\geq y$.
\end{corollary}

\subsection{Koszulity}\label{sskoszul}

Denote by $\mathscr{D}$ 
the full subcategory of the derived category $\mathcal{D}(\mathcal{O}_0^\mathbb{Z})$ 
given by the objects $\Delta_w\langle i\rangle[j]$, where
$w\in W$ and $i,j\in\mathbb{Z}$ such that 
\begin{displaymath}
i\in\{-\ell(w)-2j,-\ell(w)-2j-1\}. 
\end{displaymath}
Note that the shift put the socle of each $\Delta_w\langle i\rangle[j]$ in the ``generating diagonal'' $i\in\{-2j,-2j+1\}$. 
The group $\mathbb{Z}$  acts freely on $\mathscr{D}$ by sending $\Delta_w\langle i\rangle[j]$ to
$\Delta_w\langle i-2m\rangle[j+m]$, for $m\in\mathbb{Z}$.

Consider the category $\mathscr{D}$-mod of finite dimensional $\mathscr{D}$-modules.
The objects of this category are $\mathbb{C}$-linear functors $\mathrm{M}$
from $\mathscr{D}$ to vector spaces (over $\mathbb{C}$) such that the sum, over all 
$\mathtt{i}\in \mathscr{D}$, of the dimensions of $\mathrm{M}(\mathtt{i})$ is finite.
The morphisms in $\mathscr{D}$-mod are natural transformations of functors.

The following results generalizes \cite[Theorem~5.1]{DM} (see also  Theorem~\ref{thm7.8}).  Theorem~\ref{thm6.3-1} below shows a case that is covered by Theorem~\ref{thm7.8} but not by \cite[Theorem~5.1]{DM}.

\begin{theorem}\label{thmKoszul}
Assume that all extensions between the Verma 
modules in $\mathcal{O}_0$ are expected. Then the following assertions hold:
\begin{enumerate}[$($a$)$]
\item\label{thm7.8.1} We have an equivalence
$\mathcal{D}^*(\mathcal{O}_0^\mathbb{Z})\cong \mathcal{D}^*(\mathscr{D}\text{-}\mathrm{mod})$ where $*\in\{b,\uparrow,\downarrow\}$.
\item\label{thm7.8.2} The path algebra of $\mathscr{D}$ is Koszul and is Koszul self-dual.
\end{enumerate}
\end{theorem}

\begin{proof}
The proof essentially follows the proof of \cite[Theorem~5.1]{DM}. 

The assumption that 
all extensions between the Verma modules in $\mathcal{O}_0^{\mathbb{Z}}$ are expected says exactly 
that $\mathscr{D}$ has no self-extensions (of a nonzero degree). Since $\mathscr{D}$ generates $\mathcal{D}^b(\mathcal{O}_0^\mathbb{Z})$, the subcategory $\mathscr{D}$ gives rise to a tilting complex in the sense of Rickard (see \cite[Subsection~2.1]{DM}, where such $\mathscr{D}$ is called a tilting subset). Therefore,  Claim~\eqref{thm7.8.1}
follows from the Rickard-Morita Theorem (see \cite[Theorem~2.1]{DM}).

To show Claim~\eqref{thm7.8.2}, we note that the equivalence $\mathcal{D}^*(\mathcal{O}_0^\mathbb{Z})\to  \mathcal{D}^*(\mathscr{D}\text{-}\mathrm{mod})$ is given by $X \mapsto \Hom_{\mathcal D^{*}(\mathcal O_0^{\mathbb{Z}})}(-,X)$ where the latter functor is restricted to $\mathscr D$.
Since the Verma modules and the dual Verma modules are homologically orthogonal, simple objects in $\mathscr{D}$-mod correspond to dual Verma modules under the equivalence. 
From these we see that the quadratic dual of $\mathscr{D}$ consists of the dual Verma modules, with the similar shifts (the tops on the generating diagonal), and that the composition
\[ \mathcal{D}^\uparrow(\mathscr{D}\text{-}\mathrm{mod})\xrightarrow{\cong}\mathcal{D}^\uparrow(\mathcal{O}_0^\mathbb{Z})\xrightarrow[\cong]{-^*} \mathcal{D}^\downarrow(\mathcal{O}_0^\mathbb{Z})\xrightarrow{\cong}  \mathcal{D}^\downarrow(\mathscr{D}\text{-}\mathrm{mod}),\]
 where the first and the third functors are from \eqref{thm7.8.1} and the middle functor is the simple preserving duality, agrees with the Koszul duality functor.
Thus both the Koszulity and the Koszul self dualty follow from \cite[Theorem~30]{MOS}.
\end{proof}

\section{Combinatorics of Bruhat intervals}\label{s5}

\subsection{Equivalence classes of Bruhat intervals}\label{s5.1}

Denote by $\mathscr{I}$ the set of all pairs $(x,y)\in W^2$ such that $x\geq y$. 
Each such pair $(x,y)$ determines uniquely an interval in the Bruhat order,
denoted $[y,x]=\{z\in W\,:\, y\leq z\text{ and }z\leq x\}$. Let $\sim$
denote the minimal equivalence relations on $\mathscr{I}$ that contains all
$(x,y)\sim (xs,ys)$, where $s\in S$ is such that $\ell(x)>\ell(xs)$ and
$\ell(y)>\ell(ys)$, and all $(x,y)\sim (sx,sy)$, where $s\in S$ is such 
that $\ell(x)>\ell(sx)$ and $\ell(y)>\ell(sy)$.

\begin{proposition}\label{prop5.1-1}
If $(x,y)\sim(x',y')$, then $\exy=\e_{x',y'}$. 
\end{proposition}

\begin{proof}
In case $(x',y')=(sx,sy)$, for some $s\in S$  such 
that $\ell(x)>\ell(sx)$ and $\ell(y)>\ell(sy)$, we apply 
$\mathcal{L}\top_s$. It sends
$\Delta_{sx}$ to $\Delta_x$ and $\Delta_{sy}$ to $\Delta_y$.
On top of that, $\mathcal{L}\top_s$ is a derived equivalence and thus
induces the necessary isomorphisms between the extension spaces.
 
In case $(x',y')=(xs,ys)$, for some $s\in S$ such that $\ell(x)>\ell(xs)$ and
$\ell(y)>\ell(ys)$, we can apply the dirived equivalence 
$\mathcal{L}\mathsf{C}_s$ and argue similarly to the previous paragraph.
The claim follows. 
\end{proof}

Note that $(x,y)\sim(x',y')$ does not imply a poset isomorphism between the
Bruhat intervals $[y,x]$ and $[y',x']$ in general. For example, in type $A_3$
with simple reflections $r,s,t$
representing the following nodes of the Dynkin diagram:
$\xymatrix{r\ar@{-}[r]&s\ar@{-}[r]&t}$, we obviously have 
$(rts,e)\sim(srts,s)$. However, the boolean interval
$[e,rts]$ is not poset isomorphic to the interval $[s,srts]$.
In fact, they even have different characters as graded posets.

\subsection{Boolean and coboolean elements}\label{s5.2}

Recall that an element  $w\in W$ is called {\em boolean} provided it is a 
multiplicity-free product of simple reflections. The name is justified by
the observation that the Bruhat ideal $[e,w]$, for a boolean element $w$,
is isomorphic, as a poset, to the poset of subsets of the set of simple
reflections appearing in $w$. 

\begin{proposition}\label{prop5.2-1}
Let $x',y'\in W$ be such that $x'\geq y'$.

\begin{enumerate}[$($a$)$]
\item\label{prop5.2-1.1}
Assume that the equivalence class
of $(x',y')$ contains some $(x,y)$  with $x$ boolean. Then
$\mathrm{ext}^k(\Delta_x'\langle i\rangle,\Delta_y')$ is 
given by Fomula~\ref{eq4.4-5}.
\item\label{prop5.2-1.2}
Assume that the equivalence class
of $(x',y')$ contains some $(x,y)$ with $w_0y$ boolean. Then
$\mathrm{ext}^k(\Delta_x'\langle i\rangle,\Delta_y')$ is 
given by Fomula~\ref{eq4.4-5}.
\end{enumerate}
\end{proposition}

\begin{proof}
The two claims of the proposition are connected by the Ringel duality and the simple preserving duality $^\star$,
so it is enough to prove Claim~\eqref{prop5.2-1.2}. By Proposition~\ref{prop5.1-1}, 
it is enough to consider the case $x=x'$ and $y=y'$, i.e., $w_0y$ is boolean. (Note that the latter is if and only if $yw_0$ is boolean.)
We claim that, in this case, all KL polynomials $p_{xw_0,zw_0}$ and $p_{y,z}$, where $y\leq z\leq x$, are trivial. 
This and Proposition~\ref{klkl} proves \eqref{prop5.2-1.2}.

The claim is a well-known property of KL polynoimials (see \cite[Exercise 5.36.(e)]{BB}), which is proved, for example, as follows. If $w$ is boolean, say $w=st\cdots u$ for $s,t,\cdots,u\in S$ distinct, then the KL basis element is of the form $\underline{H}_w=\underline{H}_s\underline{H}_t\cdots \underline{H}_u$. So all KL polynoimals $p_{w',w}$, for $w'\leq w$, are trivial, and so are $p_{w_0w'w_0,w_0ww_0}$. By Kazhdan-Lusztig inversion formula (see \cite[Section 3]{KL}), the same is true for $p_{w_0w,w_0w'}$ and $p_{ww_0,w'w_0}$. 
These include all $p_{xw_0,zw_0}$ and $p_{y,z}$ since $zw_0\leq yw_0$ is boolean.
\end{proof}

\section{Special  cases}\label{s6}
\subsection{Type $A_1$}\label{s6.1}

In type $A_1$, we have $W=\{e,s\}$. The only non-zero extension of positive degree between
Verma modules is $\mathrm{ext}^1(\Delta_s\langle 1\rangle,\Delta_e)\cong \mathbb{C}$
realized in the projective module $P_s$. Here is the table for 
$\tE_{x,y}$:
\begin{displaymath}
\begin{array}{c||c|c}
x\setminus y& e& s\\
\hline\hline
e&1&0\\
\hline
s&\upsilon+\omega\upsilon^{-1}&1 
\end{array}
\end{displaymath}

\subsection{Type $A_2$}\label{s6.2}

In type $A_2$, we have $W=\{e,s,t,st,ts,w_0=sts=tst\}$.
By Corollary~\ref{cor4.5-1}, all extensions between Verma modules in this
case are given by Formula~\ref{eq4.4-5} via the coefficients of $R$-polynomials.
Here is the table for the $R$-polynomials in this case:

\resizebox{\textwidth}{!}{$
\begin{array}{c||c|c|c|c|c|c}
x\setminus y& e& s& t& st& ts& w_0\\
\hline\hline
e&1&0&0&0&0&0\\
\hline
s&v-v^{-1}&1&0&0&0&0\\
\hline
t&v-v^{-1}&0&1&0&0&0\\
\hline
st&v^2-2+v^{-2}&v-v^{-1}&v-v^{-1}&1&0&0\\
\hline
ts&v^2-2+v^{-2}&v-v^{-1}&v-v^{-1}&0&1&0\\
\hline
w_0&v^3-2v+2v^{-1}-v^{-3}&v^2-2+v^{-2}&v^2-2+v^{-2}&v-v^{-1}&v-v^{-1}&1\\
\hline
\end{array}
$}

Here is the table for $\tE$-polynomials in this case:

\resizebox{\textwidth}{!}{$
\begin{array}{c||c|c|c|c|c|c}
x\setminus y& e& s& t& st& ts& w_0\\
\hline\hline
e&1&0&0&0&0&0\\
\hline
s&\upsilon+\omega\upsilon^{-1}&1&0&0&0&0\\
\hline
t&\upsilon+\omega\upsilon^{-1}&0&1&0&0&0\\
\hline
st&\upsilon^2+2\omega+\omega^2\upsilon^{-2}&
\upsilon+\omega\upsilon^{-1}&\upsilon+\omega\upsilon^{-1}&1&0&0\\
\hline
ts&\upsilon^2+2\omega+\omega^2\upsilon^{-2}&
\upsilon+\omega\upsilon^{-1}&\upsilon+\omega\upsilon^{-1}&0&1&0\\
\hline
w_0&\upsilon^3+2\omega\upsilon+2\omega^2\upsilon^{-1}+\omega^3\upsilon^{-3}&
\upsilon^2+2\omega+\omega^2\upsilon^{-2}&
\upsilon^2+2\omega+\omega^2\upsilon^{-2}&\upsilon+\omega\upsilon^{-1}&
\upsilon+\omega\upsilon^{-1}&1\\
\hline
\end{array}
$}

\subsection{Type $A_3$}\label{s6.3}

In type $A_3$, the group $W$ is generated by the simple reflections $r,s,t$
representing the following nodes of the Dynkin diagram:
$\xymatrix{r\ar@{-}[r]&s\ar@{-}[r]&t}$. Our main result in this subsection is
the following.

\begin{theorem}\label{thm6.3-1}
In type $A_3$, all extensions between Verma modules in $\mathcal{O}_0$ 
are expected and given by Forumla~\ref{eq4.4-5}.
\end{theorem}

Please note that Theorem~\ref{thm6.3-1} does not claim that, in type $A_3$,
we are always in the situation as described by the assumptions of Corollary~\ref{cor4.4-1}.
The claim is that, regardless whether the assumptions of Corollary~\ref{cor4.4-1},
all extensions between Verma modules are given by Forumla~\ref{eq4.4-5}.

\begin{proof}
In type $A_3$, there are two non-trivial
KL-polynomials of the form $p_{e,w}$, namely:
\begin{displaymath}
p_{e,srts}=v^2+v^4\quad\text{ and }\quad p_{e,rstsr}=v^3+v^5. 
\end{displaymath}
This implies the followimg two facts:
\begin{itemize}
\item The minimal tilting coresolution $\mathcal{T}_\bullet(\Delta_e)$,
apart from the ``expected'' summands $T_w\langle\ell(w)\rangle[-\ell(w)]$, where $w\in S_4$, 
also has two additional summands: $T_{srts}\langle 2\rangle[-2]$ and 
$T_{rstsr}\langle 3\rangle[-3]$.
\item The following tilting modules have non-simple socle:
\begin{itemize}
\item the module $T_{s}$, whose expected part of the socle is 
$L_{w_0}\langle -\ell(w_0s)\rangle$, has additional socle 
$L_{w_0}\langle -\ell(w_0s)+2\rangle$,
\item the module $T_{rt}$, whose expected part of the socle is 
$L_{w_0}\langle -\ell(w_0rt)\rangle$, has additional socle 
$L_{w_0}\langle -\ell(w_0rt)+2\rangle$.
\end{itemize}
\end{itemize}
Every $T_w$ not listed above has socle $L_{w_0}\langle -\ell(w_0w)\rangle$. 
We can now collect the information about the socles of all tilting summands 
appearing in  $\mathcal{T}_\bullet(\Delta_e)$ in Figure~\ref{fig6.3-5}.
Here the expected hom dimension between $L_{w_0}=\Delta_{w_0}$ and each $\mathcal T_i(\Delta_e)$ is highlighted by {\color{magenta}magenta} 
color and the additional part is highlighted by the {\color{violet}violet} color. 
The four additional dimensions comes from the above list:
\begin{itemize}
\item the additional socle of $T_{s}\langle 1\rangle[-1]$ gives one dimension 
at the point $(1,-2)$,
\item the additional socle of $T_{rt}\langle 2\rangle[-2]$ gives one dimension 
at the point $(2,0)$,
\item the socle of the additional summand $T_{srts}\langle 2\rangle[-2]$ 
gives one dimension at the point $(2,0)$,
\item the socle of the additional summand $T_{rstsr}\langle 2\rangle[-2]$ 
gives one dimension at the point $(3,2)$,
\end{itemize}
The goal is to show that no homomorphism in a violet space gives rise to a nonzero homomorphism of complexes between the corresponding shifts of $L_{w_0}$ and $\T(\Delta_e)$. From this it follows that there is no additional homomorphism of complexes (i.e., no homomorphism of complexes that possibly gives an additional extension) between $\Delta_x$ and $\T(\Delta_e)$ since the latter would restrict to a homomorphism from the socle $L_{w_0}$ (see also the socle discussion in Subsection~\ref{s4.5}).

\begin{figure}
\resizebox{5cm}{!}{
\begin{tikzpicture}
\draw[gray, thin,  ->] (8,0) -- (-1,0) node[anchor=south west] {$\left[ a\right]$};
\draw[gray, thin,  ->] (0,-7) -- (0,7) node[anchor=north west] {$\langle b\rangle$};
\draw[gray,fill=gray] (0,0) circle (.3ex) node[anchor=north east] {\color{gray}\tiny$0$};
\draw[gray,fill=gray] (1,0) circle (.3ex) node[anchor=north east] {\color{gray}\tiny$-1$};
\draw[gray,fill=gray] (2,0) circle (.3ex) node[anchor=north east] {\color{gray}\tiny$-2$};
\draw[gray,fill=gray] (3,0) circle (.3ex) node[anchor=north east] {\color{gray}\tiny$-3$};
\draw[gray,fill=gray] (4,0) circle (.3ex) node[anchor=north east] {\color{gray}\tiny$-4$};
\draw[gray,fill=gray] (5,0) circle (.3ex) node[anchor=north east] {\color{gray}\tiny$-5$};
\draw[gray,fill=gray] (6,0) circle (.3ex) node[anchor=north east] {\color{gray}\tiny$-6$};
\draw[gray,fill=gray] (7,0) circle (.3ex) node[anchor=north east] {\color{gray}\tiny$-7$};
\draw[gray,fill=gray] (0,1) circle (.3ex) node[anchor=north east] {\color{gray}\tiny$1$};
\draw[gray,fill=gray] (1,1) circle (.3ex);
\draw[gray,fill=gray] (2,1) circle (.3ex);
\draw[gray,fill=gray] (3,1) circle (.3ex);
\draw[gray,fill=gray] (4,1) circle (.3ex);
\draw[gray,fill=gray] (5,1) circle (.3ex);
\draw[gray,fill=gray] (6,1) circle (.3ex);
\draw[gray,fill=gray] (7,1) circle (.3ex);
\draw[gray,fill=gray] (0,2) circle (.3ex) node[anchor=north east] {\color{gray}\tiny$2$};
\draw[gray,fill=gray] (1,2) circle (.3ex);
\draw[gray,fill=gray] (2,2) circle (.3ex);
\draw[gray,fill=gray] (3,2) circle (.3ex);
\draw[gray,fill=gray] (4,2) circle (.3ex);
\draw[gray,fill=gray] (5,2) circle (.3ex);
\draw[gray,fill=gray] (6,2) circle (.3ex);
\draw[gray,fill=gray] (7,2) circle (.3ex);
\draw[gray,fill=gray] (0,3) circle (.3ex) node[anchor=north east] {\color{gray}\tiny$3$};
\draw[gray,fill=gray] (1,3) circle (.3ex);
\draw[gray,fill=gray] (2,3) circle (.3ex);
\draw[gray,fill=gray] (3,3) circle (.3ex);
\draw[gray,fill=gray] (4,3) circle (.3ex);
\draw[gray,fill=gray] (5,3) circle (.3ex);
\draw[gray,fill=gray] (6,3) circle (.3ex);
\draw[gray,fill=gray] (7,3) circle (.3ex);
\draw[gray,fill=gray] (0,4) circle (.3ex) node[anchor=north east] {\color{gray}\tiny$4$};
\draw[gray,fill=gray] (1,4) circle (.3ex);
\draw[gray,fill=gray] (2,4) circle (.3ex);
\draw[gray,fill=gray] (3,4) circle (.3ex);
\draw[gray,fill=gray] (4,4) circle (.3ex);
\draw[gray,fill=gray] (5,4) circle (.3ex);
\draw[gray,fill=gray] (6,4) circle (.3ex);
\draw[gray,fill=gray] (7,4) circle (.3ex);
\draw[gray,fill=gray] (0,5) circle (.3ex) node[anchor=north east] {\color{gray}\tiny$5$};
\draw[gray,fill=gray] (1,5) circle (.3ex);
\draw[gray,fill=gray] (2,5) circle (.3ex);
\draw[gray,fill=gray] (3,5) circle (.3ex);
\draw[gray,fill=gray] (4,5) circle (.3ex);
\draw[gray,fill=gray] (5,5) circle (.3ex);
\draw[gray,fill=gray] (6,5) circle (.3ex);
\draw[gray,fill=gray] (7,5) circle (.3ex);
\draw[gray,fill=gray] (0,6) circle (.3ex) node[anchor=north east] {\color{gray}\tiny$6$};
\draw[gray,fill=gray] (1,6) circle (.3ex);
\draw[gray,fill=gray] (2,6) circle (.3ex);
\draw[gray,fill=gray] (3,6) circle (.3ex);
\draw[gray,fill=gray] (4,6) circle (.3ex);
\draw[gray,fill=gray] (5,6) circle (.3ex);
\draw[gray,fill=gray] (6,6) circle (.3ex);
\draw[gray,fill=gray] (7,6) circle (.3ex);
\draw[gray,fill=gray] (0,-1) circle (.3ex) node[anchor=north east] {\color{gray}\tiny$-1$};
\draw[gray,fill=gray] (1,-1) circle (.3ex);
\draw[gray,fill=gray] (2,-1) circle (.3ex);
\draw[gray,fill=gray] (3,-1) circle (.3ex);
\draw[gray,fill=gray] (4,-1) circle (.3ex);
\draw[gray,fill=gray] (5,-1) circle (.3ex);
\draw[gray,fill=gray] (6,-1) circle (.3ex);
\draw[gray,fill=gray] (7,-1) circle (.3ex);
\draw[gray,fill=gray] (0,-2) circle (.3ex) node[anchor=north east] {\color{gray}\tiny$-2$};
\draw[gray,fill=gray] (1,-2) circle (.3ex);
\draw[gray,fill=gray] (2,-2) circle (.3ex);
\draw[gray,fill=gray] (3,-2) circle (.3ex);
\draw[gray,fill=gray] (4,-2) circle (.3ex);
\draw[gray,fill=gray] (5,-2) circle (.3ex);
\draw[gray,fill=gray] (6,-2) circle (.3ex);
\draw[gray,fill=gray] (7,-2) circle (.3ex);
\draw[gray,fill=gray] (0,-3) circle (.3ex) node[anchor=north east] {\color{gray}\tiny$-3$};
\draw[gray,fill=gray] (1,-3) circle (.3ex);
\draw[gray,fill=gray] (2,-3) circle (.3ex);
\draw[gray,fill=gray] (3,-3) circle (.3ex);
\draw[gray,fill=gray] (4,-3) circle (.3ex);
\draw[gray,fill=gray] (5,-3) circle (.3ex);
\draw[gray,fill=gray] (6,-3) circle (.3ex);
\draw[gray,fill=gray] (7,-3) circle (.3ex);
\draw[gray,fill=gray] (0,-4) circle (.3ex) node[anchor=north east] {\color{gray}\tiny$-4$};
\draw[gray,fill=gray] (1,-4) circle (.3ex);
\draw[gray,fill=gray] (2,-4) circle (.3ex);
\draw[gray,fill=gray] (3,-4) circle (.3ex);
\draw[gray,fill=gray] (4,-4) circle (.3ex);
\draw[gray,fill=gray] (5,-4) circle (.3ex);
\draw[gray,fill=gray] (6,-4) circle (.3ex);
\draw[gray,fill=gray] (7,-4) circle (.3ex);
\draw[gray,fill=gray] (0,-5) circle (.3ex) node[anchor=north east] {\color{gray}\tiny$-5$};
\draw[gray,fill=gray] (1,-5) circle (.3ex);
\draw[gray,fill=gray] (2,-5) circle (.3ex);
\draw[gray,fill=gray] (3,-5) circle (.3ex);
\draw[gray,fill=gray] (4,-5) circle (.3ex);
\draw[gray,fill=gray] (5,-5) circle (.3ex);
\draw[gray,fill=gray] (6,-5) circle (.3ex);
\draw[gray,fill=gray] (7,-5) circle (.3ex);
\draw[gray,fill=gray] (0,-6) circle (.3ex) node[anchor=north east] {\color{gray}\tiny$-6$};
\draw[gray,fill=gray] (1,-6) circle (.3ex);
\draw[gray,fill=gray] (2,-6) circle (.3ex);
\draw[gray,fill=gray] (3,-6) circle (.3ex);
\draw[gray,fill=gray] (4,-6) circle (.3ex);
\draw[gray,fill=gray] (5,-6) circle (.3ex);
\draw[gray,fill=gray] (6,-6) circle (.3ex);
\draw[gray,fill=gray] (7,-6) circle (.3ex);
\draw[magenta,fill=magenta] (0,-6) circle (.4ex) node[anchor=west] {\color{magenta}\large$1$};
\draw[magenta,fill=magenta] (1,-4) circle (.4ex) node[anchor=west] {\color{magenta}\large$3$};
\draw[magenta,fill=magenta] (2,-2) circle (.4ex) node[anchor=west] {\color{magenta}\large$5$};
\draw[magenta,fill=magenta] (3,0) circle (.4ex) node[anchor=west] {\color{magenta}\large$6$};
\draw[magenta,fill=magenta] (4,2) circle (.4ex) node[anchor=west] {\color{magenta}\large$5$};
\draw[magenta,fill=magenta] (5,4) circle (.4ex) node[anchor=west] {\color{magenta}\large$3$};
\draw[magenta,fill=magenta] (6,6) circle (.4ex) node[anchor=west] {\color{magenta}\large$1$};
\draw[violet,fill=violet]   (1,-2) circle (.4ex) node[anchor=west] {\color{violet}\large$1$};
\draw[violet,fill=violet]   (2,0) circle (.4ex) node[anchor=west] {\color{violet}\large$2$};
\draw[violet,fill=violet]   (3,2) circle (.4ex) node[anchor=west] {\color{violet}\large$1$};
\end{tikzpicture}
}
\caption{Dimensions of socles for summands of $\mathcal{T}_\bullet(\Delta_e)$}\label{fig6.3-5}
\end{figure}
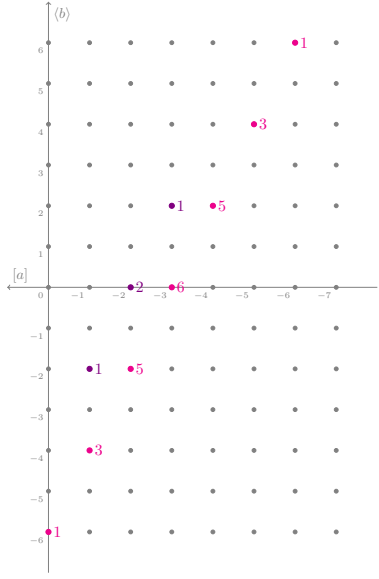

By Proposition~\ref{prop4.6-2}, the additional socle of $T_{s}\langle 1\rangle[-1]$ 
does not contribute to a homomorphism of complexes from $L_{w_0}\langle 1\rangle[-2]$
to $\mathcal{T}_\bullet(\Delta_e)$. This means that there are no additional
first extensions to $\Delta_e$ (this is a general fact, see \cite[Theorem~32]{Ma})
from any Verma modules. Therefore the additional violet 
dimension $1$ at the point $(-1,-2)$  decreases the value $5$ at the point
$(-2,-2)$ by $1$, resulting in the dimension $4$.

By Koszul-Ringel duality, the fact that we have
$\mathrm{ext}^1(\Delta_{w_0}\langle -2\rangle,\Delta_e)=0$ implies that we have
$\mathrm{ext}^3(\Delta_{w_0}\langle 2\rangle,\Delta_e)=0$.
Therefore, any homomorphism of complexes from $\Delta_{w_0}\langle 2\rangle[-3]$
to $\mathcal{T}_\bullet(\Delta_e)$ is homotopic to zero. Since no homotopies 
between these two complexes are
possible (as $\mathrm{hom}(\Delta_{w_0}\langle 2\rangle,\mathcal{T}_2(\Delta_e))=0$),
it follows that the composition of the differential in 
$\mathcal{T}_\bullet(\Delta_e)$
homomorphism $\Delta_{w_0}\langle 2\rangle to \mathcal{T}_3(\Delta_e))$ is non-zero. 
Therefore the only relevance of the additional violet 
dimension $1$ at the point $(-3,2)$ is that it decreases the value $5$ at the point
$(-4,2)$ by $1$ resulting in $4$, which is a coefficients of the $R$-polynomial.

It remains to deal with the violet point $(-2,0)$. Similarly to the above,
using Proposition~\ref{prop4.6-2}, we obtain that the restriction of
the differential in $\mathcal{T}_\bullet(\Delta_e)$ to
the additional socle of  $T_{rt}\langle 2\rangle[-2]$ is non-zero.
This takes care of one dimension at the violet point $(2,0)$.
The second dimension at this point corresponds to the usual (simple) socle
of the additional summand $T_{srts}\langle 2\rangle[-2]$. We now argue that
the image of the restriction of the differential in $\mathcal{T}_\bullet(\Delta_e)$ to
this socle component is linearly independent to the image of the restriction of
the differential in $\mathcal{T}_\bullet(\Delta_e)$ to
the additional socle of  $T_{rt}\langle 2\rangle[-2]$.

Our argument is very much ad hoc, we use the explicit description of $\Delta_e$
as given in \cite[Appendix~A]{St}. From it we see that the differential in
$\mathcal{T}_\bullet(\Delta_e)$ restricts to a non-zero map from 
$T_{srts}\langle 2\rangle[-2]$ to $T_{sts}\langle 3\rangle[-3]$ and, 
at the same time, this differential restricts to the zero map from
$T_{rt}\langle 2\rangle[-2]$ to $T_{sts}\langle 3\rangle[-3]$. It remains to show that
any non-zero map from $T_{srts}\langle 2\rangle[-2]$ to $T_{sts}\langle 3\rangle[-3]$ 
is injective (i.e., does not kill the socle). Since the domain has a standard filtration
and the codomain has a costandrd filtration, a non-zero map from 
$T_{srts}\langle 2\rangle[-2]$ to $T_{sts}\langle 3\rangle[-3]$ must be a linear 
combination of maps lifted from some standard subquotient of 
$T_{srts}\langle 2\rangle[-2]$ to some costandrd subquotient of 
$T_{sts}\langle 3\rangle[-3]$. There is only one such pair which appears
with correct shifts: $\Delta_{srts}\langle 2\rangle[-2]$ for $T_{srts}\langle 2\rangle[-2]$
and $\nabla_{srts}\langle 2\rangle[-2]$ for $T_{sts}\langle 3\rangle[-3]$.
Therefore, a non-zero homomorphism from $T_{srts}\langle 2\rangle[-2]$ to 
$T_{sts}\langle 3\rangle[-3]$ is a lift of a homomorphism from 
$\Delta_{srts}\langle 2\rangle[-2]$ to $T_{sts}\langle 3\rangle[-3]$.
Since any homomorphism from a Verma module to a tilting module is injective,
a non-zero map from $\Delta_{srts}\langle 2\rangle[-2]$ to $T_{sts}\langle 3\rangle[-3]$
is injective. Since the socles of $T_{srts}\langle 2\rangle[-2]$
and $\Delta_{srts}\langle 2\rangle[-2]$ coincide, we get the claim.

The arguments above show that, in the category of complexes, the only 
homomorphisms from Verma modules to $\mathcal{T}_\bullet(\Delta_e)$ are
with expected shifts. Since all $\mathcal{T}_\bullet(\Delta_y)$ are subcomplexes
of $\mathcal{T}_\bullet(\Delta_e)$ (up to some shifts in both homological and
grading), it follows that the only possible extensions 
between Verma modules are expected extensions. 
The fact
that they are given by Forumla~\ref{eq4.4-5} follows from the graded Delorme formula.
This completes the proof.
\end{proof}

\subsection{Some first extensions between Verma modules in type $A$}\label{s6.4}

In this subsection we use the results of \cite{KMM1} to construct many additional 
first extensions between Verma modules in type $A$. We refer to 
\cite{KMM1} for the details of the facts recalled below. We assume that $W=S_n$ with the
Dynkin diagram 
\[\xymatrix{s_1\ar@{-}[r]&s_2\ar@{-}[r]&\dots\ar@{-}[r]&s_{n-1}}.\]
For $i,j\in\{1,2,\dots,n-1\}$, we denote by $\hat{w}_{i,j}$ the following element:
\begin{displaymath}
\hat{w}_{i,j}=
\begin{cases}
s_is_{i-1}\dots s_j,& j\leq i;\\
s_is_{i+1}\dots s_j,& j> i.
\end{cases}
\end{displaymath}
Note that $\hat{w}_{i,i}=s_i$. By construction, $\hat{w}_{i,j}$ has left 
descent $s_i$ and right descent $s_j$. We also denote by $w_{i,j}$ the 
element $\hat{w}_{i,n-j}w_0$. Dually, the element  $w_{i,j}$ has left
ascent $s_i$ and right ascent $s_j$. The set $\{w_{i,j}\}$ is exactly the
penultimate two-sided KL cell in $S_n$.

Each simple $L_{w_{i,j}}$ is graded multiplicity-free in $\Delta_e$. 
In fact, $[\Delta_e:L_{w_{i,j}}\langle -m\rangle]=1$ if and only if
$m\in\{\ell(w_{i,j}),\ell(w_{i,j})-2,\ell(w_{i,j})-4,\dots,
\ell(w_{i,j})-2q_{i,j}\}$ where
\begin{displaymath}
q_{i,j}=\min\{n-1-i,n-1-j,i-1,j-1\}.
\end{displaymath}
Note that the ungraded multiplicity of $L_{w_{i,j}}$ in $\Delta_e$
equals $1+q_{i,j}$. 

Recall that an element of $S_n$ is called bigrassmannian provided that 
it has a unique left descent and a unique right descent
(for example, all $\hat{w}_{i,j}$ are bigrassmannian).
There are exactly $1+q_{i,j}$ bigrassmannian elements
in $S_n$ with left descent $s_i$ and right descent $s_j$. They form a
chain with respect to the Bruhat order and hence are in natural bijection,
denoted $\Phi_{i,j}$, with the graded simple subquotients of $\Delta_e$ 
isomorphic, up to a shift, to  $L_{w_{i,j}}$ (ordered by increasing 
graded shifts).  In particular, the element $\hat{w}_{i,j}$ corresponds
to $L_{w_{i,j}}\langle -(\ell(w_{i,j})-2q_{i,j})\rangle$.

For $w\in S_n$, denote by $\mathbf{BM}_w$ the set of all  Bruhat maximal
elements in the set of all bigrassmannian elements of the Bruhat interval
$[e,w]$.

\begin{proposition}\label{prop6.4-1}
Let $w\in S_n$ and let $u\in \mathbf{BM}_w$ have left descent $s_i$ and
right descent $s_j$. If  $\Phi_{i,j}(u)=L_{w_{i,j}}\langle m\rangle$
and $w_{i,j}\geq w$,
then $\mathrm{ext}^1(\Delta_{w_{i,j}}\langle m\rangle,\Delta_{w})\neq 0$.
\end{proposition}

\begin{proof}
Applying $\mathrm{hom}({}_-, \Delta_{w})$ to the short exact sequence
\begin{displaymath}
0\to \mathrm{Ker}\to \Delta \langle m\rangle\to L_{w_{i,j}}\langle m\rangle\to 0,
\end{displaymath}
we get the exact sequence
\begin{multline*}
0\to  \mathrm{hom}(L_{w_{i,j}}\langle m\rangle, \Delta_{w})\to
\mathrm{hom}(\Delta_{w_{i,j}}\langle m\rangle, \Delta_{w})\to
\mathrm{hom}(\mathrm{Ker}\langle m\rangle, \Delta_{w})\to \\ \to
\mathrm{ext}^1(L_{w_{i,j}}\langle m\rangle, \Delta_{w})\to
\mathrm{ext}^1(\Delta_{w_{i,j}}\langle m\rangle, \Delta_{w}).
\end{multline*}
Here $\mathrm{hom}(L_{w_{i,j}}\langle m\rangle, \Delta_{w})=0$  since
$w_{i,j}\neq w_0$. 

The above implies that the map
\begin{displaymath}
\mathrm{hom}(\Delta_{w_{i,j}}\langle m\rangle, \Delta_{w})\to
\mathrm{hom}(\mathrm{Ker}\langle m\rangle, \Delta_{w})
\end{displaymath}
is an inclusion. Under the assumption $w_{i,j}\geq w$, this map,
in fact, is an isomorphism. Indeed, we even have
\begin{displaymath}
\mathbb{C}\cong\mathrm{Hom}_{\mathcal{O}}(\Delta_{w_{i,j}}, \Delta_{w})\cong
\mathrm{Hom}_{\mathcal{O}}(\mathrm{Ker}, \Delta_{w}).
\end{displaymath}
Here the first isomorphism is a consequence of $w_{i,j}\geq w$ while the second
one is given by the restriction together with the fact that the last 
hom-space is one-dimensional since both involved modules have isomorphic 
simple socle which, moreover, has multiplicity one in both modules.

This implies that the map 
\begin{displaymath}
\mathrm{ext}^1(L_{w_{i,j}}\langle m\rangle, \Delta_{w})\to
\mathrm{ext}^1(\Delta_{w_{i,j}}\langle m\rangle, \Delta_{w})
\end{displaymath}
is injective. Since $\mathrm{ext}^1(L_{w_{i,j}}\langle m\rangle, \Delta_{w})\neq 0$
by \cite[Corollary~2]{KMM1}, the claim follows.
\end{proof}

We note that the non-zero extension in
$\mathrm{ext}^1(\Delta_{w_{i,j}}\langle m\rangle,\Delta_{w})$ produced by
Proposition~\ref{prop6.4-1} is expected if and only if $-m=\ell(w_{i,j})-2$.
In all other cases, we have an additional extension.

\begin{corollary}\label{cor6.4-2}
For $i\in\{1,2,\dots,n-1\}$,  consider the element $w_{i,n-i}=s_iw_0$.
Let $w\in S_n$ and let $m = 2- (\ell(w_{i,n-i})-\ell(w))$ (the expected degree of  $ext^1$ between $\Delta_{w_{i,n-i}}$ and $\Delta_{w}$). 
\begin{enumerate}[$($a$)$]
\item\label{cor6.4-2.1} If $\Phi_{i,n-i}\inv (L_{w_{i,n-i}}\langle m'\rangle )\not\in BM_w$ for all $m'\neq m$, then we have
\[\dim\mathrm{Ext}^1_{\mathcal{O}}(\Delta_{w_{i,n-i}},\Delta_{w})=
\dim\mathrm{ext}^1(\Delta_{w_{i,n-i}}\langle  m\rangle,\Delta_{w})=|r^{(m)}_{w_{i,n-i},w}|.\]
\item\label{cor6.4-2.2} 
If 
$\Phi_{i,n-i}\inv (L_{w_{i,n-i}}\langle m' \rangle )\in BM_w$ for some $m'\neq m$, then we have
\begin{displaymath}
\mathrm{Ext}^1_{\mathcal{O}}(\Delta_{w_{i,n-i}},\Delta_{w})=
\mathrm{ext}^1(\Delta_{w_{i,n-i}}\langle m'\rangle,\Delta_{w})
\oplus 
\mathrm{ext}^1(\Delta_{w_{i,n-i}}\langle m\rangle,\Delta_{w}).
\end{displaymath}
Moreover, we have $\mathrm{ext}^1(\Delta_{w_{i,n-i}}\langle m'\rangle,\Delta_{w})\cong\mathbb{C}$, while 
$\dim\mathrm{ext}^1(\Delta_{w_{i,n-i}}\langle m\rangle,\Delta_{w})=|r^{(m)}_{w_{i,n-i},w}|$.
\end{enumerate}
\end{corollary}

\begin{proof}
Since $w_{i,n-i}=s_iw_0$, we have the exact sequence
\begin{displaymath}
0\to  \Delta_{w_0}\langle -1\rangle\to 
\Delta_{w_{i,n-i}}\to L_{w_{i,n-i}}\to 0.
\end{displaymath}
It induces the exact sequence
\begin{displaymath}
\mathrm{Ext}^1_{\mathcal{O}}(L_{w_{i,n-i}}, \Delta_{w})\to
\mathrm{Ext}^1_{\mathcal{O}}(\Delta_{w_{i,n-i}}, \Delta_{w})\to
\mathrm{Ext}^1_{\mathcal{O}}(\Delta_{w_{0}}, \Delta_{w}).
\end{displaymath}
By \cite[Theorem~32]{Ma}, any element in $\mathrm{Ext}^1_{\mathcal{O}}(\Delta_{w_{0}}, \Delta_{w})$
is expected. By \cite[Corollary~2]{KMM1}, the dimension of 
$\mathrm{Ext}^1_{\mathcal{O}}(L_{w_{i,n-i}}, \Delta_{w})$ is at most $1$.
This naturally splits our consideration in two cases: 
if $-m=\ell(w_{i,j})-2$ and $-m\neq \ell(w_{i,j})-2$. In the first case, all
first extensions are expected and thus given by formula Formula~\ref{eq4.4-5}.
This is exactly Claim~\eqref{cor6.4-2.1}.

Now, assume that $-m\neq \ell(w_{i,j})-2$.
The proof of Proposition~\ref{prop6.4-1} constructs an embedding from
$\mathrm{ext}^1(L_{w_{i,n-i}}\langle m\rangle, \Delta_{w})$
to $\mathrm{ext}^1(\Delta_{w_{i,n-i}}\langle m\rangle, \Delta_{w})$.
We have the vanishing $\mathrm{ext}^1(\Delta_{w_{0}}\langle m-1\rangle, \Delta_{w})=0$ by
\cite[Theorem~32]{Ma}. This implies that the above embedding is, in fact, an isomorphism.
Now Claim~\eqref{cor6.4-2.2} follows from \cite[Corollary~2]{KMM1} and the
observation that all expected extensions are given by Formula~\ref{eq4.4-5}.
\end{proof}

\begin{example}\label{example6.4-3}
{\rm 
Consider $S_{2n}$, for $n>2$, and $w_{n,n}=s_nw_0$. Let $w=s_n$, which is
bigrassmannian. Then the socle of the module $\Delta_e/\Delta_{s_n}$  is ismorphic to the module
$L_{w_{n,n}}\langle -(n(2n-1)-1-2(n-1))\rangle$. Hence 
$\mathrm{Ext}^1_{\mathcal{O}}(\Delta_{w_{i,j}},\Delta_{w})$ has the expected part
of dimension $2n-1$ (corresponding to 
$\mathrm{ext}^1(\Delta_{w_{n,n}}\langle -(n(2n-1)-3)\rangle,\Delta_{s_2})$) 
and also the one-dimensional additional part 
$\mathrm{ext}^1(\Delta_{w_{n,n}}\langle -(n(2n-1)-1-2(n-1))\rangle,\Delta_{s_n})$.
}
\end{example}

\subsection{Some additional higher extensions}\label{s6.5}

\hk{This section presents a few ways to find additional, in the sense of Subsection~\ref{s4.5}, higher extensions. }

We start with the first type of examples observed (in an ungraded setting) by Boe \cite{Bo} in disproving the Gabber-Joseph conjecture.
\begin{example}
If the coefficients $r_{x,y}^{(k)}$ do not alternate in sign, then Theorem~\ref{thm3.2-1} implies that there is an additional extension between $\Delta_x$ and $\Delta_y$. (Here we cannot determine the $i$ such that there are additional $i$-th extensions.) Computer computation of $r_{x,y}$ provides many such examples. We record two here: 

In type $D_4$, the coefficients of  $r_{e,w_0}$ are 
\[[1, -4, 7, -8, 6, 0, -4, 0, 6, -8, 7, -4, 1].\]

The coefficients of $r_{e,w_0}$ in type $E_7$ are
\[[-1, 7, -22, 42, -57, 63, -65, 71, -87, 113, -137, 127, -55, -47, 111, -137, 173,\] 
\[-171,23, 223, -399, 505, -708, 1052, -1396, 1580, -1530, 1302, -984, 456, 430,\]
\[-1250, 1250, 
-430, -456, 984, -1302, 1530, -1580, 1396, -1052, 708, -505, 399,\]
\[-223, -23, 171, -173,137, -111, 47, 55, -127, 137, -113, 87, -71,65, -63, 57,\] 
\[ -42, 22, -7, 1].\] 
\end{example}

The second type of examples are also, more or less, combinatorial, this time depending on computations of KL polynomials rather than $R$-polynomials and using Proposition~\ref{klkl} rather than Theorem~\ref{thm3.2-1}.
\begin{example}
Let $W$ be of $B_3$ with the labeling $\begin{tikzpicture}[scale=0.4,baseline=-3]
\protect\draw (0 cm,0) -- (-2 cm,0);
\protect\draw (-2 cm,0.1cm) -- (-4 cm,0.1 cm);
\protect\draw (-2 cm,-0.1cm) -- (-4 cm,-0.1 cm);
\protect\draw[fill=white] (0 cm, 0 cm) circle (.15cm) node[above=1pt]{\scriptsize $2$};
\protect\draw[fill=white] (-2 cm, 0 cm) circle (.15cm) node[above=1pt]{\scriptsize $1$};
\protect\draw[fill=white] (-4 cm, 0 cm) circle (.15cm) node[above=1pt]{\scriptsize $0$};
\end{tikzpicture}$.
We claim that 
\begin{enumerate}
    \item either $\Ext^2(\Delta_{w_0},\Delta_e)$ or $\Ext^2(\Delta_{w_0},\Delta_{s_0})$ contains additional extensions;
    \item either $\Ext^3(\Delta_{w_0},\Delta_e)$ or $\Ext^3(\Delta_{w_0},\Delta_{s_0})$ contains additional extensions;
    \item either $\Ext^4(\Delta_{w_0},\Delta_e)$ or $\Ext^4(\Delta_{w_0},\Delta_{s_0})$ contains additional extensions.
\end{enumerate}
A computation of KL polynomials, together with Proposition~\ref{klkl}, shows that the dimensions of
$\hom(\Delta_{w_0}\langle b\rangle,\mathcal{T}_{a}(\Delta_e)) $ and $\hom(\Delta_{w_0}\langle b\rangle,\mathcal{T}_{a-1}(\Delta_{s_0}))$ are as in Figure~\ref{B3fig}. Recall that the embedding $\mathcal{T}_{\bullet-1}(\Delta_{s_0})\to \mathcal{T}_{\bullet}(\Delta_{e})$ induces an embedding 
\[\iota:\hom(\Delta_{w_0}\langle b\rangle,\mathcal{T}_{\bullet-1}(\Delta_{s_0}))\to \hom(\Delta_{w_0}\langle b\rangle,\mathcal{T}_{\bullet}(\Delta_e)).\] 
The claims additional extensions arise from the violet (i.e., additional) coordinates $(a,b)$ where the dimension difference at $(a-1,b)$ is greater than the difference at $(a,b)$ in Figure~\ref{B3fig}.  
We explain the details for the Claim (1). The same argument applies to the other claims.

Suppose the second extension between $\Delta_{w_0}$ and $\Delta_e$ is expected. Then the map \[d\circ - :\hom(\Delta_{w_0}\langle -1\rangle,\mathcal{T}_{2}(\Delta_e)) \to \hom(\Delta_{w_0}\langle -1\rangle,\mathcal{T}_{3}(\Delta_e))\] is injective, where $d$ is the (relevant restriction of) differential of $\T(\Delta_e)$. Let $V$ be its 3-dimensional image in $\hom(\Delta_{w_0}\langle -1\rangle,\mathcal{T}_{3}(\Delta_e))$. 
The embedding $\iota$ restricts to an isomorphism 
\[\iota:\hom(\Delta_{w_0}\langle -1\rangle,\mathcal{T}_{3-1}(\Delta_{s_0}))\xrightarrow[]{\simeq} \hom(\Delta_{w_0}\langle -1\rangle,\mathcal{T}_{3}(\Delta_e)),\] also denoted by $\iota$, since both spaces are of dimension $7$.
We obtain a 3-dimensional subspace $\iota^{-1}(V)$ in the 7-dimensional space $\hom(\Delta_{w_0}\langle -1\rangle,\mathcal{T}_{3-1}(\Delta_{s_0}))$ which corresponds to morphisms of complexes from $\Delta_{w_0}\langle-1\rangle$ to $\T(\Delta_{s_0})$. At most (in fact exactly) two dimensional subspace of $\iota^{-1}$ is homotopic to zero, because the dimension of at $(2,-1)$ is two. Thus the rest contributes to a nonzero element in $\ext^2(\Delta_{w_0}\langle-1\rangle,\Delta_{s_0})$ which is additional. \qed

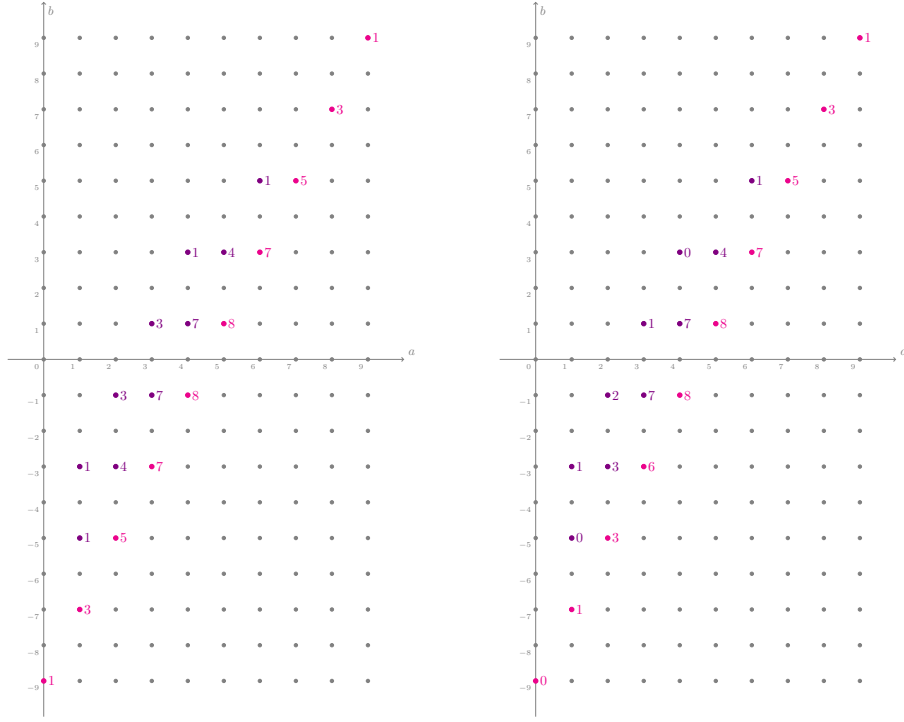
\begin{figure}
\resizebox{12cm}{!}{
\begin{tikzpicture}
\draw[gray, thin,  ->] (-1,0) -- (10,0) node[anchor=south west] {$a$};
\draw[gray, thin,  ->] (0,-10) -- (0,10) node[anchor=north west] {$b$};
\draw[gray,fill=gray] (0,0) circle (.3ex) node[anchor=north east] {\color{gray}\tiny$0$};
\draw[gray,fill=gray] (1,0) circle (.3ex) node[anchor=north east] {\color{gray}\tiny$1$};
\draw[gray,fill=gray] (2,0) circle (.3ex) node[anchor=north east] {\color{gray}\tiny$2$};
\draw[gray,fill=gray] (3,0) circle (.3ex) node[anchor=north east] {\color{gray}\tiny$3$};
\draw[gray,fill=gray] (4,0) circle (.3ex) node[anchor=north east] {\color{gray}\tiny$4$};
\draw[gray,fill=gray] (5,0) circle (.3ex) node[anchor=north east] {\color{gray}\tiny$5$};
\draw[gray,fill=gray] (6,0) circle (.3ex) node[anchor=north east] {\color{gray}\tiny$6$};
\draw[gray,fill=gray] (7,0) circle (.3ex) node[anchor=north east] {\color{gray}\tiny$7$};
\draw[gray,fill=gray] (8,0) circle (.3ex) node[anchor=north east] {\color{gray}\tiny$8$};
\draw[gray,fill=gray] (9,0) circle (.3ex) node[anchor=north east] {\color{gray}\tiny$9$};
\draw[gray,fill=gray] (0,1) circle (.3ex) node[anchor=north east] {\color{gray}\tiny$1$};
\draw[gray,fill=gray] (1,1) circle (.3ex);
\draw[gray,fill=gray] (2,1) circle (.3ex);
\draw[gray,fill=gray] (3,1) circle (.3ex);
\draw[gray,fill=gray] (4,1) circle (.3ex);
\draw[gray,fill=gray] (5,1) circle (.3ex);
\draw[gray,fill=gray] (6,1) circle (.3ex);
\draw[gray,fill=gray] (7,1) circle (.3ex);
\draw[gray,fill=gray] (8,1) circle (.3ex);
\draw[gray,fill=gray] (9,1) circle (.3ex);
\draw[gray,fill=gray] (0,2) circle (.3ex) node[anchor=north east] {\color{gray}\tiny$2$};
\draw[gray,fill=gray] (1,2) circle (.3ex);
\draw[gray,fill=gray] (2,2) circle (.3ex);
\draw[gray,fill=gray] (3,2) circle (.3ex);
\draw[gray,fill=gray] (4,2) circle (.3ex);
\draw[gray,fill=gray] (5,2) circle (.3ex);
\draw[gray,fill=gray] (6,2) circle (.3ex);
\draw[gray,fill=gray] (7,2) circle (.3ex);
\draw[gray,fill=gray] (8,2) circle (.3ex);
\draw[gray,fill=gray] (9,2) circle (.3ex);
\draw[gray,fill=gray] (0,3) circle (.3ex) node[anchor=north east] {\color{gray}\tiny$3$};
\draw[gray,fill=gray] (1,3) circle (.3ex);
\draw[gray,fill=gray] (2,3) circle (.3ex);
\draw[gray,fill=gray] (3,3) circle (.3ex);
\draw[gray,fill=gray] (4,3) circle (.3ex);
\draw[gray,fill=gray] (5,3) circle (.3ex);
\draw[gray,fill=gray] (6,3) circle (.3ex);
\draw[gray,fill=gray] (7,3) circle (.3ex);
\draw[gray,fill=gray] (8,3) circle (.3ex);
\draw[gray,fill=gray] (9,3) circle (.3ex);
\draw[gray,fill=gray] (0,4) circle (.3ex) node[anchor=north east] {\color{gray}\tiny$4$};
\draw[gray,fill=gray] (1,4) circle (.3ex);
\draw[gray,fill=gray] (2,4) circle (.3ex);
\draw[gray,fill=gray] (3,4) circle (.3ex);
\draw[gray,fill=gray] (4,4) circle (.3ex);
\draw[gray,fill=gray] (5,4) circle (.3ex);
\draw[gray,fill=gray] (6,4) circle (.3ex);
\draw[gray,fill=gray] (7,4) circle (.3ex);
\draw[gray,fill=gray] (8,4) circle (.3ex);
\draw[gray,fill=gray] (9,4) circle (.3ex);
\draw[gray,fill=gray] (0,5) circle (.3ex) node[anchor=north east] {\color{gray}\tiny$5$};
\draw[gray,fill=gray] (1,5) circle (.3ex);
\draw[gray,fill=gray] (2,5) circle (.3ex);
\draw[gray,fill=gray] (3,5) circle (.3ex);
\draw[gray,fill=gray] (4,5) circle (.3ex);
\draw[gray,fill=gray] (5,5) circle (.3ex);
\draw[gray,fill=gray] (6,5) circle (.3ex);
\draw[gray,fill=gray] (7,5) circle (.3ex);
\draw[gray,fill=gray] (8,5) circle (.3ex);
\draw[gray,fill=gray] (9,5) circle (.3ex);
\draw[gray,fill=gray] (0,6) circle (.3ex) node[anchor=north east] {\color{gray}\tiny$6$};
\draw[gray,fill=gray] (1,6) circle (.3ex);
\draw[gray,fill=gray] (2,6) circle (.3ex);
\draw[gray,fill=gray] (3,6) circle (.3ex);
\draw[gray,fill=gray] (4,6) circle (.3ex);
\draw[gray,fill=gray] (5,6) circle (.3ex);
\draw[gray,fill=gray] (6,6) circle (.3ex);
\draw[gray,fill=gray] (7,6) circle (.3ex);
\draw[gray,fill=gray] (8,6) circle (.3ex);
\draw[gray,fill=gray] (9,6) circle (.3ex);
\draw[gray,fill=gray] (0,7) circle (.3ex) node[anchor=north east] {\color{gray}\tiny$7$};
\draw[gray,fill=gray] (1,7) circle (.3ex);
\draw[gray,fill=gray] (2,7) circle (.3ex);
\draw[gray,fill=gray] (3,7) circle (.3ex);
\draw[gray,fill=gray] (4,7) circle (.3ex);
\draw[gray,fill=gray] (5,7) circle (.3ex);
\draw[gray,fill=gray] (6,7) circle (.3ex);
\draw[gray,fill=gray] (7,7) circle (.3ex);
\draw[gray,fill=gray] (8,7) circle (.3ex);
\draw[gray,fill=gray] (9,7) circle (.3ex);
\draw[gray,fill=gray] (0,8) circle (.3ex) node[anchor=north east] {\color{gray}\tiny$8$};
\draw[gray,fill=gray] (1,8) circle (.3ex);
\draw[gray,fill=gray] (2,8) circle (.3ex);
\draw[gray,fill=gray] (3,8) circle (.3ex);
\draw[gray,fill=gray] (4,8) circle (.3ex);
\draw[gray,fill=gray] (5,8) circle (.3ex);
\draw[gray,fill=gray] (6,8) circle (.3ex);
\draw[gray,fill=gray] (7,8) circle (.3ex);
\draw[gray,fill=gray] (8,8) circle (.3ex);
\draw[gray,fill=gray] (9,8) circle (.3ex);
\draw[gray,fill=gray] (0,9) circle (.3ex) node[anchor=north east] {\color{gray}\tiny$9$};
\draw[gray,fill=gray] (1,9) circle (.3ex);
\draw[gray,fill=gray] (2,9) circle (.3ex);
\draw[gray,fill=gray] (3,9) circle (.3ex);
\draw[gray,fill=gray] (4,9) circle (.3ex);
\draw[gray,fill=gray] (5,9) circle (.3ex);
\draw[gray,fill=gray] (6,9) circle (.3ex);
\draw[gray,fill=gray] (7,9) circle (.3ex);
\draw[gray,fill=gray] (8,9) circle (.3ex);
\draw[gray,fill=gray] (9,9) circle (.3ex);
\draw[gray,fill=gray] (0,-1) circle (.3ex) node[anchor=north east] {\color{gray}\tiny$-1$};
\draw[gray,fill=gray] (1,-1) circle (.3ex);
\draw[gray,fill=gray] (2,-1) circle (.3ex);
\draw[gray,fill=gray] (3,-1) circle (.3ex);
\draw[gray,fill=gray] (4,-1) circle (.3ex);
\draw[gray,fill=gray] (5,-1) circle (.3ex);
\draw[gray,fill=gray] (6,-1) circle (.3ex);
\draw[gray,fill=gray] (7,-1) circle (.3ex);
\draw[gray,fill=gray] (8,-1) circle (.3ex);
\draw[gray,fill=gray] (9,-1) circle (.3ex);
\draw[gray,fill=gray] (0,-2) circle (.3ex) node[anchor=north east] {\color{gray}\tiny$-2$};
\draw[gray,fill=gray] (1,-2) circle (.3ex);
\draw[gray,fill=gray] (2,-2) circle (.3ex);
\draw[gray,fill=gray] (3,-2) circle (.3ex);
\draw[gray,fill=gray] (4,-2) circle (.3ex);
\draw[gray,fill=gray] (5,-2) circle (.3ex);
\draw[gray,fill=gray] (6,-2) circle (.3ex);
\draw[gray,fill=gray] (7,-2) circle (.3ex);
\draw[gray,fill=gray] (8,-2) circle (.3ex);
\draw[gray,fill=gray] (9,-2) circle (.3ex);
\draw[gray,fill=gray] (0,-3) circle (.3ex) node[anchor=north east] {\color{gray}\tiny$-3$};
\draw[gray,fill=gray] (1,-3) circle (.3ex);
\draw[gray,fill=gray] (2,-3) circle (.3ex);
\draw[gray,fill=gray] (3,-3) circle (.3ex);
\draw[gray,fill=gray] (4,-3) circle (.3ex);
\draw[gray,fill=gray] (5,-3) circle (.3ex);
\draw[gray,fill=gray] (6,-3) circle (.3ex);
\draw[gray,fill=gray] (7,-3) circle (.3ex);
\draw[gray,fill=gray] (8,-3) circle (.3ex);
\draw[gray,fill=gray] (9,-3) circle (.3ex);
\draw[gray,fill=gray] (0,-4) circle (.3ex) node[anchor=north east] {\color{gray}\tiny$-4$};
\draw[gray,fill=gray] (1,-4) circle (.3ex);
\draw[gray,fill=gray] (2,-4) circle (.3ex);
\draw[gray,fill=gray] (3,-4) circle (.3ex);
\draw[gray,fill=gray] (4,-4) circle (.3ex);
\draw[gray,fill=gray] (5,-4) circle (.3ex);
\draw[gray,fill=gray] (6,-4) circle (.3ex);
\draw[gray,fill=gray] (7,-4) circle (.3ex);
\draw[gray,fill=gray] (8,-4) circle (.3ex);
\draw[gray,fill=gray] (9,-4) circle (.3ex);
\draw[gray,fill=gray] (0,-5) circle (.3ex) node[anchor=north east] {\color{gray}\tiny$-5$};
\draw[gray,fill=gray] (1,-5) circle (.3ex);
\draw[gray,fill=gray] (2,-5) circle (.3ex);
\draw[gray,fill=gray] (3,-5) circle (.3ex);
\draw[gray,fill=gray] (4,-5) circle (.3ex);
\draw[gray,fill=gray] (5,-5) circle (.3ex);
\draw[gray,fill=gray] (6,-5) circle (.3ex);
\draw[gray,fill=gray] (7,-5) circle (.3ex);
\draw[gray,fill=gray] (8,-5) circle (.3ex);
\draw[gray,fill=gray] (9,-5) circle (.3ex);
\draw[gray,fill=gray] (0,-6) circle (.3ex) node[anchor=north east] {\color{gray}\tiny$-6$};
\draw[gray,fill=gray] (1,-6) circle (.3ex);
\draw[gray,fill=gray] (2,-6) circle (.3ex);
\draw[gray,fill=gray] (3,-6) circle (.3ex);
\draw[gray,fill=gray] (4,-6) circle (.3ex);
\draw[gray,fill=gray] (5,-6) circle (.3ex);
\draw[gray,fill=gray] (6,-6) circle (.3ex);
\draw[gray,fill=gray] (7,-6) circle (.3ex);
\draw[gray,fill=gray] (8,-6) circle (.3ex);
\draw[gray,fill=gray] (9,-6) circle (.3ex);
\draw[gray,fill=gray] (1,-7) circle (.3ex);
\draw[gray,fill=gray] (2,-7) circle (.3ex);
\draw[gray,fill=gray] (3,-7) circle (.3ex);
\draw[gray,fill=gray] (4,-7) circle (.3ex);
\draw[gray,fill=gray] (5,-7) circle (.3ex);
\draw[gray,fill=gray] (6,-7) circle (.3ex);
\draw[gray,fill=gray] (7,-7) circle (.3ex);
\draw[gray,fill=gray] (8,-7) circle (.3ex);
\draw[gray,fill=gray] (9,-7) circle (.3ex);
\draw[gray,fill=gray] (0,-7) circle (.3ex) node[anchor=north east] {\color{gray}\tiny$-7$};
\draw[gray,fill=gray] (1,-8) circle (.3ex);
\draw[gray,fill=gray] (2,-8) circle (.3ex);
\draw[gray,fill=gray] (3,-8) circle (.3ex);
\draw[gray,fill=gray] (4,-8) circle (.3ex);
\draw[gray,fill=gray] (5,-8) circle (.3ex);
\draw[gray,fill=gray] (6,-8) circle (.3ex);
\draw[gray,fill=gray] (7,-8) circle (.3ex);
\draw[gray,fill=gray] (8,-8) circle (.3ex);
\draw[gray,fill=gray] (9,-8) circle (.3ex);
\draw[gray,fill=gray] (0,-8) circle (.3ex) node[anchor=north east] {\color{gray}\tiny$-8$};
\draw[gray,fill=gray] (1,-9) circle (.3ex);
\draw[gray,fill=gray] (2,-9) circle (.3ex);
\draw[gray,fill=gray] (3,-9) circle (.3ex);
\draw[gray,fill=gray] (4,-9) circle (.3ex);
\draw[gray,fill=gray] (5,-9) circle (.3ex);
\draw[gray,fill=gray] (6,-9) circle (.3ex);
\draw[gray,fill=gray] (7,-9) circle (.3ex);
\draw[gray,fill=gray] (8,-9) circle (.3ex);
\draw[gray,fill=gray] (9,-9) circle (.3ex);
\draw[gray,fill=gray] (0,-9) circle (.3ex) node[anchor=north east] {\color{gray}\tiny$-9$};
\draw[magenta,fill=magenta] (0,-9) circle (.4ex) node[anchor=west] {\color{magenta}\large$1$};
\draw[magenta,fill=magenta] (1,-7) circle (.4ex) node[anchor=west] {\color{magenta}\large$3$};
\draw[magenta,fill=magenta] (2,-5) circle (.4ex) node[anchor=west] {\color{magenta}\large$5$};
\draw[magenta,fill=magenta] (3,-3) circle (.4ex) node[anchor=west] {\color{magenta}\large$7$};
\draw[magenta,fill=magenta] (4,-1) circle (.4ex) node[anchor=west] {\color{magenta}\large$8$};
\draw[magenta,fill=magenta] (5,1) circle (.4ex) node[anchor=west] {\color{magenta}\large$8$};
\draw[magenta,fill=magenta] (6,3) circle (.4ex) node[anchor=west] {\color{magenta}\large$7$};
\draw[magenta,fill=magenta] (7,5) circle (.4ex) node[anchor=west] {\color{magenta}\large$5$};
\draw[magenta,fill=magenta] (8,7) circle (.4ex) node[anchor=west] {\color{magenta}\large$3$};
\draw[magenta,fill=magenta] (9,9) circle (.4ex) node[anchor=west] {\color{magenta}\large$1$};

\draw[violet,fill=violet]   (1,-5) circle (.4ex) node[anchor=west] {\color{violet}\large$1$};
\draw[violet,fill=violet]   (1,-3) circle (.4ex) node[anchor=west] {\color{violet}\large$1$};
\draw[violet,fill=violet]   (2,-3) circle (.4ex) node[anchor=west] {\color{violet}\large$4$};
\draw[violet,fill=violet]   (2,-1) circle (.4ex) node[anchor=west] {\color{violet}\large$3$};
\draw[violet,fill=violet]   (3,-1) circle (.4ex) node[anchor=west] {\color{violet}\large$7$};
\draw[violet,fill=violet]   (3,1) circle (.4ex) node[anchor=west] {\color{violet}\large$3$};
\draw[violet,fill=violet]   (4,1) circle (.4ex) node[anchor=west] {\color{violet}\large$7$};
\draw[violet,fill=violet]   (4,3) circle (.4ex) node[anchor=west] {\color{violet}\large$1$};
\draw[violet,fill=violet]   (5,3) circle (.4ex) node[anchor=west] {\color{violet}\large$4$};
\draw[violet,fill=violet]   (6,5) circle (.4ex) node[anchor=west] {\color{violet}\large$1$};
\end{tikzpicture}
\quad\quad\quad\quad\quad\quad
\begin{tikzpicture}
\draw[gray, thin,  ->] (-1,0) -- (10,0) node[anchor=south west] {$a$};
\draw[gray, thin,  ->] (0,-10) -- (0,10) node[anchor=north west] {$b$};
\draw[gray,fill=gray] (0,0) circle (.3ex) node[anchor=north east] {\color{gray}\tiny$0$};
\draw[gray,fill=gray] (1,0) circle (.3ex) node[anchor=north east] {\color{gray}\tiny$1$};
\draw[gray,fill=gray] (2,0) circle (.3ex) node[anchor=north east] {\color{gray}\tiny$2$};
\draw[gray,fill=gray] (3,0) circle (.3ex) node[anchor=north east] {\color{gray}\tiny$3$};
\draw[gray,fill=gray] (4,0) circle (.3ex) node[anchor=north east] {\color{gray}\tiny$4$};
\draw[gray,fill=gray] (5,0) circle (.3ex) node[anchor=north east] {\color{gray}\tiny$5$};
\draw[gray,fill=gray] (6,0) circle (.3ex) node[anchor=north east] {\color{gray}\tiny$6$};
\draw[gray,fill=gray] (7,0) circle (.3ex) node[anchor=north east] {\color{gray}\tiny$7$};
\draw[gray,fill=gray] (8,0) circle (.3ex) node[anchor=north east] {\color{gray}\tiny$8$};
\draw[gray,fill=gray] (9,0) circle (.3ex) node[anchor=north east] {\color{gray}\tiny$9$};
\draw[gray,fill=gray] (0,1) circle (.3ex) node[anchor=north east] {\color{gray}\tiny$1$};
\draw[gray,fill=gray] (1,1) circle (.3ex);
\draw[gray,fill=gray] (2,1) circle (.3ex);
\draw[gray,fill=gray] (3,1) circle (.3ex);
\draw[gray,fill=gray] (4,1) circle (.3ex);
\draw[gray,fill=gray] (5,1) circle (.3ex);
\draw[gray,fill=gray] (6,1) circle (.3ex);
\draw[gray,fill=gray] (7,1) circle (.3ex);
\draw[gray,fill=gray] (8,1) circle (.3ex);
\draw[gray,fill=gray] (9,1) circle (.3ex);
\draw[gray,fill=gray] (0,2) circle (.3ex) node[anchor=north east] {\color{gray}\tiny$2$};
\draw[gray,fill=gray] (1,2) circle (.3ex);
\draw[gray,fill=gray] (2,2) circle (.3ex);
\draw[gray,fill=gray] (3,2) circle (.3ex);
\draw[gray,fill=gray] (4,2) circle (.3ex);
\draw[gray,fill=gray] (5,2) circle (.3ex);
\draw[gray,fill=gray] (6,2) circle (.3ex);
\draw[gray,fill=gray] (7,2) circle (.3ex);
\draw[gray,fill=gray] (8,2) circle (.3ex);
\draw[gray,fill=gray] (9,2) circle (.3ex);
\draw[gray,fill=gray] (0,3) circle (.3ex) node[anchor=north east] {\color{gray}\tiny$3$};
\draw[gray,fill=gray] (1,3) circle (.3ex);
\draw[gray,fill=gray] (2,3) circle (.3ex);
\draw[gray,fill=gray] (3,3) circle (.3ex);
\draw[gray,fill=gray] (4,3) circle (.3ex);
\draw[gray,fill=gray] (5,3) circle (.3ex);
\draw[gray,fill=gray] (6,3) circle (.3ex);
\draw[gray,fill=gray] (7,3) circle (.3ex);
\draw[gray,fill=gray] (8,3) circle (.3ex);
\draw[gray,fill=gray] (9,3) circle (.3ex);
\draw[gray,fill=gray] (0,4) circle (.3ex) node[anchor=north east] {\color{gray}\tiny$4$};
\draw[gray,fill=gray] (1,4) circle (.3ex);
\draw[gray,fill=gray] (2,4) circle (.3ex);
\draw[gray,fill=gray] (3,4) circle (.3ex);
\draw[gray,fill=gray] (4,4) circle (.3ex);
\draw[gray,fill=gray] (5,4) circle (.3ex);
\draw[gray,fill=gray] (6,4) circle (.3ex);
\draw[gray,fill=gray] (7,4) circle (.3ex);
\draw[gray,fill=gray] (8,4) circle (.3ex);
\draw[gray,fill=gray] (9,4) circle (.3ex);
\draw[gray,fill=gray] (0,5) circle (.3ex) node[anchor=north east] {\color{gray}\tiny$5$};
\draw[gray,fill=gray] (1,5) circle (.3ex);
\draw[gray,fill=gray] (2,5) circle (.3ex);
\draw[gray,fill=gray] (3,5) circle (.3ex);
\draw[gray,fill=gray] (4,5) circle (.3ex);
\draw[gray,fill=gray] (5,5) circle (.3ex);
\draw[gray,fill=gray] (6,5) circle (.3ex);
\draw[gray,fill=gray] (7,5) circle (.3ex);
\draw[gray,fill=gray] (8,5) circle (.3ex);
\draw[gray,fill=gray] (9,5) circle (.3ex);
\draw[gray,fill=gray] (0,6) circle (.3ex) node[anchor=north east] {\color{gray}\tiny$6$};
\draw[gray,fill=gray] (1,6) circle (.3ex);
\draw[gray,fill=gray] (2,6) circle (.3ex);
\draw[gray,fill=gray] (3,6) circle (.3ex);
\draw[gray,fill=gray] (4,6) circle (.3ex);
\draw[gray,fill=gray] (5,6) circle (.3ex);
\draw[gray,fill=gray] (6,6) circle (.3ex);
\draw[gray,fill=gray] (7,6) circle (.3ex);
\draw[gray,fill=gray] (8,6) circle (.3ex);
\draw[gray,fill=gray] (9,6) circle (.3ex);
\draw[gray,fill=gray] (0,7) circle (.3ex) node[anchor=north east] {\color{gray}\tiny$7$};
\draw[gray,fill=gray] (1,7) circle (.3ex);
\draw[gray,fill=gray] (2,7) circle (.3ex);
\draw[gray,fill=gray] (3,7) circle (.3ex);
\draw[gray,fill=gray] (4,7) circle (.3ex);
\draw[gray,fill=gray] (5,7) circle (.3ex);
\draw[gray,fill=gray] (6,7) circle (.3ex);
\draw[gray,fill=gray] (7,7) circle (.3ex);
\draw[gray,fill=gray] (8,7) circle (.3ex);
\draw[gray,fill=gray] (9,7) circle (.3ex);
\draw[gray,fill=gray] (0,8) circle (.3ex) node[anchor=north east] {\color{gray}\tiny$8$};
\draw[gray,fill=gray] (1,8) circle (.3ex);
\draw[gray,fill=gray] (2,8) circle (.3ex);
\draw[gray,fill=gray] (3,8) circle (.3ex);
\draw[gray,fill=gray] (4,8) circle (.3ex);
\draw[gray,fill=gray] (5,8) circle (.3ex);
\draw[gray,fill=gray] (6,8) circle (.3ex);
\draw[gray,fill=gray] (7,8) circle (.3ex);
\draw[gray,fill=gray] (8,8) circle (.3ex);
\draw[gray,fill=gray] (9,8) circle (.3ex);
\draw[gray,fill=gray] (0,9) circle (.3ex) node[anchor=north east] {\color{gray}\tiny$9$};
\draw[gray,fill=gray] (1,9) circle (.3ex);
\draw[gray,fill=gray] (2,9) circle (.3ex);
\draw[gray,fill=gray] (3,9) circle (.3ex);
\draw[gray,fill=gray] (4,9) circle (.3ex);
\draw[gray,fill=gray] (5,9) circle (.3ex);
\draw[gray,fill=gray] (6,9) circle (.3ex);
\draw[gray,fill=gray] (7,9) circle (.3ex);
\draw[gray,fill=gray] (8,9) circle (.3ex);
\draw[gray,fill=gray] (9,9) circle (.3ex);
\draw[gray,fill=gray] (0,-1) circle (.3ex) node[anchor=north east] {\color{gray}\tiny$-1$};
\draw[gray,fill=gray] (1,-1) circle (.3ex);
\draw[gray,fill=gray] (2,-1) circle (.3ex);
\draw[gray,fill=gray] (3,-1) circle (.3ex);
\draw[gray,fill=gray] (4,-1) circle (.3ex);
\draw[gray,fill=gray] (5,-1) circle (.3ex);
\draw[gray,fill=gray] (6,-1) circle (.3ex);
\draw[gray,fill=gray] (7,-1) circle (.3ex);
\draw[gray,fill=gray] (8,-1) circle (.3ex);
\draw[gray,fill=gray] (9,-1) circle (.3ex);
\draw[gray,fill=gray] (0,-2) circle (.3ex) node[anchor=north east] {\color{gray}\tiny$-2$};
\draw[gray,fill=gray] (1,-2) circle (.3ex);
\draw[gray,fill=gray] (2,-2) circle (.3ex);
\draw[gray,fill=gray] (3,-2) circle (.3ex);
\draw[gray,fill=gray] (4,-2) circle (.3ex);
\draw[gray,fill=gray] (5,-2) circle (.3ex);
\draw[gray,fill=gray] (6,-2) circle (.3ex);
\draw[gray,fill=gray] (7,-2) circle (.3ex);
\draw[gray,fill=gray] (8,-2) circle (.3ex);
\draw[gray,fill=gray] (9,-2) circle (.3ex);
\draw[gray,fill=gray] (0,-3) circle (.3ex) node[anchor=north east] {\color{gray}\tiny$-3$};
\draw[gray,fill=gray] (1,-3) circle (.3ex);
\draw[gray,fill=gray] (2,-3) circle (.3ex);
\draw[gray,fill=gray] (3,-3) circle (.3ex);
\draw[gray,fill=gray] (4,-3) circle (.3ex);
\draw[gray,fill=gray] (5,-3) circle (.3ex);
\draw[gray,fill=gray] (6,-3) circle (.3ex);
\draw[gray,fill=gray] (7,-3) circle (.3ex);
\draw[gray,fill=gray] (8,-3) circle (.3ex);
\draw[gray,fill=gray] (9,-3) circle (.3ex);
\draw[gray,fill=gray] (0,-4) circle (.3ex) node[anchor=north east] {\color{gray}\tiny$-4$};
\draw[gray,fill=gray] (1,-4) circle (.3ex);
\draw[gray,fill=gray] (2,-4) circle (.3ex);
\draw[gray,fill=gray] (3,-4) circle (.3ex);
\draw[gray,fill=gray] (4,-4) circle (.3ex);
\draw[gray,fill=gray] (5,-4) circle (.3ex);
\draw[gray,fill=gray] (6,-4) circle (.3ex);
\draw[gray,fill=gray] (7,-4) circle (.3ex);
\draw[gray,fill=gray] (8,-4) circle (.3ex);
\draw[gray,fill=gray] (9,-4) circle (.3ex);
\draw[gray,fill=gray] (0,-5) circle (.3ex) node[anchor=north east] {\color{gray}\tiny$-5$};
\draw[gray,fill=gray] (1,-5) circle (.3ex);
\draw[gray,fill=gray] (2,-5) circle (.3ex);
\draw[gray,fill=gray] (3,-5) circle (.3ex);
\draw[gray,fill=gray] (4,-5) circle (.3ex);
\draw[gray,fill=gray] (5,-5) circle (.3ex);
\draw[gray,fill=gray] (6,-5) circle (.3ex);
\draw[gray,fill=gray] (7,-5) circle (.3ex);
\draw[gray,fill=gray] (8,-5) circle (.3ex);
\draw[gray,fill=gray] (9,-5) circle (.3ex);
\draw[gray,fill=gray] (0,-6) circle (.3ex) node[anchor=north east] {\color{gray}\tiny$-6$};
\draw[gray,fill=gray] (1,-6) circle (.3ex);
\draw[gray,fill=gray] (2,-6) circle (.3ex);
\draw[gray,fill=gray] (3,-6) circle (.3ex);
\draw[gray,fill=gray] (4,-6) circle (.3ex);
\draw[gray,fill=gray] (5,-6) circle (.3ex);
\draw[gray,fill=gray] (6,-6) circle (.3ex);
\draw[gray,fill=gray] (7,-6) circle (.3ex);
\draw[gray,fill=gray] (8,-6) circle (.3ex);
\draw[gray,fill=gray] (9,-6) circle (.3ex);
\draw[gray,fill=gray] (1,-7) circle (.3ex);
\draw[gray,fill=gray] (2,-7) circle (.3ex);
\draw[gray,fill=gray] (3,-7) circle (.3ex);
\draw[gray,fill=gray] (4,-7) circle (.3ex);
\draw[gray,fill=gray] (5,-7) circle (.3ex);
\draw[gray,fill=gray] (6,-7) circle (.3ex);
\draw[gray,fill=gray] (7,-7) circle (.3ex);
\draw[gray,fill=gray] (8,-7) circle (.3ex);
\draw[gray,fill=gray] (9,-7) circle (.3ex);
\draw[gray,fill=gray] (0,-7) circle (.3ex) node[anchor=north east] {\color{gray}\tiny$-7$};
\draw[gray,fill=gray] (1,-8) circle (.3ex);
\draw[gray,fill=gray] (2,-8) circle (.3ex);
\draw[gray,fill=gray] (3,-8) circle (.3ex);
\draw[gray,fill=gray] (4,-8) circle (.3ex);
\draw[gray,fill=gray] (5,-8) circle (.3ex);
\draw[gray,fill=gray] (6,-8) circle (.3ex);
\draw[gray,fill=gray] (7,-8) circle (.3ex);
\draw[gray,fill=gray] (8,-8) circle (.3ex);
\draw[gray,fill=gray] (9,-8) circle (.3ex);
\draw[gray,fill=gray] (0,-8) circle (.3ex) node[anchor=north east] {\color{gray}\tiny$-8$};
\draw[gray,fill=gray] (1,-9) circle (.3ex);
\draw[gray,fill=gray] (2,-9) circle (.3ex);
\draw[gray,fill=gray] (3,-9) circle (.3ex);
\draw[gray,fill=gray] (4,-9) circle (.3ex);
\draw[gray,fill=gray] (5,-9) circle (.3ex);
\draw[gray,fill=gray] (6,-9) circle (.3ex);
\draw[gray,fill=gray] (7,-9) circle (.3ex);
\draw[gray,fill=gray] (8,-9) circle (.3ex);
\draw[gray,fill=gray] (9,-9) circle (.3ex);
\draw[gray,fill=gray] (0,-9) circle (.3ex) node[anchor=north east] {\color{gray}\tiny$-9$};
\draw[magenta,fill=magenta] (0,-9) circle (.4ex) node[anchor=west] {\color{magenta}\large$0$};
\draw[magenta,fill=magenta] (1,-7) circle (.4ex) node[anchor=west] {\color{magenta}\large$1$};
\draw[magenta,fill=magenta] (2,-5) circle (.4ex) node[anchor=west] {\color{magenta}\large$3$};
\draw[magenta,fill=magenta] (3,-3) circle (.4ex) node[anchor=west] {\color{magenta}\large$6$};
\draw[magenta,fill=magenta] (4,-1) circle (.4ex) node[anchor=west] {\color{magenta}\large$8$};
\draw[magenta,fill=magenta] (5,1) circle (.4ex) node[anchor=west] {\color{magenta}\large$8$};
\draw[magenta,fill=magenta] (6,3) circle (.4ex) node[anchor=west] {\color{magenta}\large$7$};
\draw[magenta,fill=magenta] (7,5) circle (.4ex) node[anchor=west] {\color{magenta}\large$5$};
\draw[magenta,fill=magenta] (8,7) circle (.4ex) node[anchor=west] {\color{magenta}\large$3$};
\draw[magenta,fill=magenta] (9,9) circle (.4ex) node[anchor=west] {\color{magenta}\large$1$};

\draw[violet,fill=violet]   (1,-5) circle (.4ex) node[anchor=west] {\color{violet}\large$0$};
\draw[violet,fill=violet]   (1,-3) circle (.4ex) node[anchor=west] {\color{violet}\large$1$};
\draw[violet,fill=violet]   (2,-3) circle (.4ex) node[anchor=west] {\color{violet}\large$3$};
\draw[violet,fill=violet]   (2,-1) circle (.4ex) node[anchor=west] {\color{violet}\large$2$};
\draw[violet,fill=violet]   (3,-1) circle (.4ex) node[anchor=west] {\color{violet}\large$7$};
\draw[violet,fill=violet]   (3,1) circle (.4ex) node[anchor=west] {\color{violet}\large$1$};
\draw[violet,fill=violet]   (4,1) circle (.4ex) node[anchor=west] {\color{violet}\large$7$};
\draw[violet,fill=violet]   (4,3) circle (.4ex) node[anchor=west] {\color{violet}\large$0$};
\draw[violet,fill=violet]   (5,3) circle (.4ex) node[anchor=west] {\color{violet}\large$4$};
\draw[violet,fill=violet]   (6,5) circle (.4ex) node[anchor=west] {\color{violet}\large$1$};
\end{tikzpicture}
}
\caption{$\dim \hom(\Delta_{w_0}\langle b\rangle,\mathcal{T}_a(\Delta_e)) $ versus $\dim \hom(\Delta_{w_0}\langle b\rangle,\mathcal{T}_{a-1}(\Delta_{s_0})) $}\label{B3fig}
\end{figure}
\end{example}

\hk{Another way to construct additional extensions is to use \cite{KMM1} and Subsection~\ref{s6.4}, as in the following proposition.
}

\begin{proposition}\label{prop6.5-1}
In the setup of Example~\ref{example6.4-3}, the Yoneda product of an additional
element in $\mathrm{ext}^1(\Delta_{w_{n,n}}\langle -(n(2n-1)-1-2(n-1))\rangle,\Delta_{s_n})$
and an additional element in 
$\mathrm{ext}^1(\Delta_{s_n}\langle 1\rangle,\Delta_e)$
gives an additional element in
$$\mathrm{ext}^2(\Delta_{w_{n,n}}\langle -(n(2n-1)-2(n-1))\rangle,\Delta_{e}).$$
\end{proposition}

This implies that 
$\mathrm{ext}^2(\Delta_{w_{n,n}}\langle -(n(2n-1)-2(n-1))\rangle,\Delta_{e})\neq 0$
and gives an example of an additional second extension.

\begin{proof}
From the Koszul-Ringel self-duality, it follows that
the complex $\mathcal{T}_\bullet(\Delta_{s_n})[-1]\langle 1\rangle$ is a subcomplex
of the complex $\mathcal{T}_\bullet(\Delta_e)$. This inclusion corresponds precisely
to an additional element in $\mathrm{ext}^1(\Delta_{s_n}\langle 1\rangle,\Delta_e)$.

A non-zero element in 
$\mathrm{ext}^1(\Delta_{w_{n,n}}\langle -(n(2n-1)-1-2(n-1))\rangle,\Delta_{s_n})$
corresponds to a non-zero homomorphism in the homotopy category of complexes from 
the singleton complex $\Delta_{w_{n,n}}[-1]\langle -(n(2n-1)-1-2(n-1))\rangle$ to
$\mathcal{T}_\bullet(\Delta_{s_n})$. Therefore, to prove the claim it is enough to
show that the map from $\Delta_{w_{n,n}}[-2]\langle -(n(2n-1)-2(n-1))\rangle$
to $\mathcal{T}_\bullet(\Delta_e)$ induced by the inclusion of 
$\mathcal{T}_\bullet(\Delta_{s_n})[-1]\langle 1\rangle$ to 
$\mathcal{T}_\bullet(\Delta_e)$ is not homotopic to zero.

To prove this, it is enough to show that there are no non-zero homomorphisms from
$\Delta_{w_{n,n}}\langle -(n(2n-1)-2(n-1))\rangle$ to any indecomposable direct
summand of $\mathcal{T}_1(\Delta_e)$ outside of 
$\mathcal{T}_0(\Delta_{s_n})[-1]\langle 1\rangle$.
 
These direct summands are exactly the modules $T_{s_i}\langle 1\rangle$, where
$i\neq n$. Since $T_{s_i}$ is a tilting module, a non-zero homomorphism from 
$\Delta_{w_{n,n}}\langle -(n(2n-1)-2(n-1))\rangle$ to $T_{s_i}\langle 1\rangle$ 
exists if and only if $\nabla_{w_{n,n}}\langle -(n(2n-1)-2(n-1))\rangle$ is a
subquotient of a dual Verma flag of $T_{s_i}\langle 1\rangle$. Using
$\top_{w_0}$, this is equivalent to $\Delta_{s_n}\langle -(n(2n-1)-2(n-1))\rangle$
being a subquotient of a dual Verma flag of $P_{w_0s_i}\langle 1\rangle$.
By the BGG reciprocity, this is equivalent to $L_{w_0s_i}\langle -(n(2n-1)-2(n-1))\rangle$
being a composition subquotient of $\Delta_{s_n}$.

Note that $w_0s_i$ belongs to the penultimate KL-cell in the terminology of 
\cite{KMM1}. All graded simple penultimate subquotients of $\Delta_e$
are described in \cite[Proposition~12]{KMM1}. From \cite[Theorem~1]{KMM1}
it follows that the socle of the module $\Delta_e/(\Delta_{s_n}\langle -1\rangle)$
is the unique penultimate subquotient of of this module and that it occurs in the
minimal possible degree (in $\Delta_e$) among all other penultimate subquotients
of $\Delta_e$. It follows that any other simple subquotient of $\Delta_{s_n}$
of the form $L_{w_0s_i}$ appears in $\Delta_e$ in a strictly higher degree compared to
the degree of the socle of $\Delta_e/(\Delta_{s_n}\langle -1\rangle)$.
Going back via the BGG reciprocity and the Ringel duality, we get exactly the claim 
that any subquotient $\nabla_{w_{n,n}}\langle d\rangle$ 
of a dual Verma flag of $T_{s_i}\langle 1\rangle$ must be shifted strictly more than by
$-(n(2n-1)-2(n-1))$. This completes the proof.
\end{proof}

\section{Extensions between singular and between parabolic Verma modules}\label{s7}

We briefly discuss generalizations of the previous sections to singular and parabolic categories $\mathcal O$.

\subsection{Singular blocks of $\mathcal{O}$}\label{s7.1}

Thanks to Soergel's combinatorial description of blocks of category $\mathcal{O}$,
see \cite{So}, it is known that every block of $\mathcal{O}$ is equivalent to 
an integral (but, in general, singular) block of $\mathcal{O}$ (however, possibly,
for a different Lie algebra). Therefore the complete version of the problem to 
describe extensions between Verma modules must address the case of singular 
integral blocks.

Let $\mathfrak{p}$ denote a parabolic subalgebra of $\mathfrak{g}$ containing
the Borel subalgebra $\mathfrak{h}\oplus\mathfrak{n}_+$. The subalgebra 
$\mathfrak{p}$ is uniquely determined by a subset of simple roots, or,
equivalently, by the corresponding parabolic subgroup $W^{\mathfrak{p}}$
of $W$. Let $\mathbf{R}^\mathfrak{p}_{\mathrm{short}}$ denote the set of 
the shortest coset representatives in $W/W^{\mathfrak{p}}$.

Let $\lambda$ be a dominant integral weight such that $W^{\mathfrak{p}}$ is
exactly the dot-stabilizer of $\lambda$. Consider the block $\mathcal{O}_\lambda$
of $\mathcal{O}$ containing $L(\lambda)$. Then the simple objects in 
$\mathcal{O}_\lambda$ are $\{L(w\cdot\lambda)\,:\, w\in 
\mathbf{R}^\mathfrak{p}_{\mathrm{short}}\}$. Similarly to the regular case,
we also have the corresponding projective, injective, Verma, dual Verma and
tilting modules and their graded versions.

\subsection{Regular blocks of parabolic category $\mathcal{O}$}\label{s7.2}

Associated to our choice of $\mathfrak{p}$, one also has the parabolic 
category $\mathcal{O}^\mathfrak{p}$ introduced in \cite{RC}. It is defined
as the full subcategory of $\mathcal{O}$ consisting of all objects, the 
action of $U(\mathfrak{p})$ on which is locally finite.

Let $\mathbf{L}^\mathfrak{p}_{\mathrm{short}}$ denote the set of 
the shortest coset representatives in $W^{\mathfrak{p}}\backslash W$.
Then the category $\mathcal{O}_0^\mathfrak{p}$ is the Serre subcategory of
$\mathcal{O}_0$ generated by all $L_w$ such that  
$w\in \mathbf{L}^\mathfrak{p}_{\mathrm{short}}$. We use the superscript
$\mathfrak{p}$ to denote structural objects in $\mathcal{O}_0^\mathfrak{p}$.
In particular, for $w\in \mathbf{L}^\mathfrak{p}_{\mathrm{short}}$, we denote
by  $P^\mathfrak{p}_w$ the indecomposable projective cover of 
$L^\mathfrak{p}_w=L_w$ in $\mathcal{O}_0^\mathfrak{p}$ and so on.
The category $\mathcal{O}_0^\mathfrak{p}$ inherits a graded lift from that 
for $\mathcal{O}_0$.

\subsection{Koszul-Ringel duality}\label{s7.3}

For a fixed parabolic subalgebra $\mathfrak{p}$ as above and singular dominant
integral $\lambda$ with dot-stabilizer $W^\mathfrak{p}$, the combination of
Koszul and Ringel dualities, together with the autoequivalence given by 
the conjugation with $w_0$, see \cite{BGS,So1-2,Ma2}, gives rise to the equivalence
\begin{displaymath}
\mathcal{D}^b\big((\mathcal{O}_\lambda)^\mathbb{Z}\big)\cong
\mathcal{D}^b\big((\mathcal{O}_0^\mathfrak{p})^\mathbb{Z}\big)
\end{displaymath}
which sends $\Delta(w\cdot \lambda)$ to $\Delta^{\mathfrak{p}}_{w^{-1}}$,
where $w\in \mathbf{R}^\mathfrak{p}_{\mathrm{short}}$. 

In particular, this implies that 
\begin{displaymath}
\mathrm{ext}^k(\Delta(x\cdot \lambda),\Delta(y\cdot\lambda)\langle j\rangle)\cong
\mathrm{ext}^{k+j}(\Delta^\mathfrak{p}_{x^{-1}},
\Delta^\mathfrak{p}_{y^{-1}}\langle -j\rangle)
\end{displaymath}
and thus the problem to determine all extensions between singular Verma modules
is equivalent to the problem to detrmine all extensions between regular parabolic
Verma modules.

\subsection{Singular and parabolic $R$-polynomials}\label{s7.4}

Consider the usual $\mathbb{Z}[v,v^{-1}]$-structure on the  Grothendieck group 
$\mathrm{Gr}(\mathcal{O}_\lambda^{\mathbb{Z}})$. Similarly to the regular
case, the group $\mathrm{Gr}(\mathcal{O}_\lambda^{\mathbb{Z}})$ has
various bases given by the classes of simple, standard, costandard,
projective, injective and tilting objects.

The {\em singular $R$-polynomials} $\{sr_{x,y}\,:\,x,y\in W\}$ are 
defined as the entries of the transformation matrix between the standard 
and the costandard basesin $\mathrm{Gr}(\mathcal{O}_\lambda^{\mathbb{Z}})$, i.e.:
\begin{displaymath}
[\Delta(y\cdot \lambda)]=\sum_{x\in 
\mathbf{R}^\mathfrak{p}_{\mathrm{short}}}sr_{x,y}[\nabla(x\cdot\lambda)], 
\text{ for all } y\in \mathbf{R}^\mathfrak{p}_{\mathrm{short}}.
\end{displaymath}
Note that $sr_{x,y}\in\mathbb{Z}[v,v^{-1}]$, by definition. 
For $x,y\in \mathbf{R}^\mathfrak{p}_{\mathrm{short}}$ and 
$k\in\mathbb{Z}$, we denote by $sr_{x,y}^{(k)}$ the coefficient
at $v^k$ in $sr_{x,y}$.

Let $w_0^\mathfrak{p}$ denotes the longest element in $W^\mathfrak{p}$.
The connection between the usual and the singular $R$-polynomials is clarified
by the following:

\begin{lemma}\label{lem7.4-1}
For $x,y\in \mathbf{R}^\mathfrak{p}_{\mathrm{short}}$, we have
\begin{displaymath}
sr_{x,y}=\sum_{w\in W^\mathfrak{p}}r_{x,yw}v^{\ell(w)-2\ell(w_0^\mathfrak{p})}. 
\end{displaymath}
\end{lemma}

\begin{proof}
Let $\theta_{\lambda}^{\mathrm{on}}$ be the translation functor to the $\lambda$-wall,
that is the unique indecomposable projective functor in the sense of \cite{BG} which
sends $\Delta_e$ to $\Delta_{\lambda}$. Then $\theta_{\lambda}^{\mathrm{on}}$ sends
$\Delta_y$ to $\Delta(y\cdot \lambda)$. Furthermore, for 
$y\in \mathbf{R}^\mathfrak{p}_{\mathrm{long}}$ and $w\in W^\mathfrak{p}$, we have 
\begin{displaymath}
\theta_{\lambda}^{\mathrm{on}}\nabla_{yw}\cong
\nabla(y\cdot\lambda)\langle 2\ell(w_0^\mathfrak{p})-\ell(w)\rangle.
\end{displaymath}
Since $\theta_{\lambda}^{\mathrm{on}}$ is exact, the
claim now follows from the definitions. 
\end{proof}

Consider the usual $\mathbb{Z}[v,v^{-1}]$-structure on the  Grothendieck group 
$\mathrm{Gr}\big((\mathcal{O}_0^\mathfrak{p})^{\mathbb{Z}}\big)$. Similarly to the regular
case, the group $\mathrm{Gr}\big((\mathcal{O}_0^\mathfrak{p})^{\mathbb{Z}}\big)$ has
various bases given by the classes of simple, standard, costandard,
projective, injective and tilting objects.

The {\em parabolic $R$-polynomials} $\{pr_{x,y}\,:\,x,y\in W\}$ are 
defined as the entries of the transformation matrix between the standard 
and the costandard bases in 
$\mathrm{Gr}\big((\mathcal{O}_0^\mathfrak{p})^{\mathbb{Z}}\big)$, i.e.:
\begin{displaymath}
[\Delta_y^\mathfrak{p}]=\sum_{x\in 
\mathbf{L}^\mathfrak{p}_{\mathrm{short}}}pr_{x,y}[\nabla_x^\mathfrak{p}], 
\text{ for all } y\in \mathbf{L}^\mathfrak{p}_{\mathrm{short}}.
\end{displaymath}
Note that $pr_{x,y}\in\mathbb{Z}[v,v^{-1}]$, by definition. 
For $x,y\in \mathbf{L}^\mathfrak{p}_{\mathrm{short}}$ and 
$k\in\mathbb{Z}$, we denote by $pr_{x,y}^{(k)}$ the coefficient
at $v^k$ in $pr_{x,y}$.

As $\Delta_{w_0}=\nabla_{w_0}$, we have
\begin{equation}\label{eq7.4-1}
pr_{x,w_0^\mathfrak{p}w_0}=
\begin{cases}
1,& x=w_0^\mathfrak{p}w_0;\\
0,& \text{otherwise}.
\end{cases}
\end{equation}

For $w\in \mathbf{L}^\mathfrak{p}_{\mathrm{short}}$ and $s\in S$ such that 
$ws\not\in \mathbf{L}^\mathfrak{p}_{\mathrm{short}}$, we have 
$\theta_s\Delta^\mathfrak{p}_w=\theta_s\nabla^\mathfrak{p}_w=0$. If
$ws\in \mathbf{L}^\mathfrak{p}_{\mathrm{short}}$ and
$ws>w$, then we have
\begin{displaymath}
[\theta_s\Delta^\mathfrak{p}_w]=v[\theta_s\Delta^\mathfrak{p}_{ws}]=
[\Delta^\mathfrak{p}_{ws}]+v[\Delta^\mathfrak{p}
_w]\,\text{ and }\,
[\theta_s\nabla^\mathfrak{p}_w]=
v^{-1}[\theta_s\nabla^\mathfrak{p}_{ws}]=[\nabla_{ws}]+v^{-1}[\nabla^\mathfrak{p}_w].
\end{displaymath}
From this, we have the following recursive formula for parabolic $R$-polynomials:
For $x,y\in \mathbf{L}^\mathfrak{p}_{\mathrm{short}}$ and $s\in S$ 
such that $ys<y$ and $ys\in \mathbf{L}^\mathfrak{p}_{\mathrm{short}}$, we have:
\begin{equation}\label{eq7.4-2}
pr_{x,ys}=
\begin{cases}
pr_{xs,y},& xs<x\text{ and }xs\in \mathbf{L}^\mathfrak{p}_{\mathrm{short}};\\ 
pr_{xs,y}+ (v^{-1}-v)r_{x,y},& xs>x\text{ and }
xs\in \mathbf{L}^\mathfrak{p}_{\mathrm{short}};\\
-v\, pr_{x,y},& xs\not\in \mathbf{L}^\mathfrak{p}_{\mathrm{short}}.
\end{cases}
\end{equation}
Together, Formulae~\eqref{eq7.4-1} and \eqref{eq7.4-2} determine the
family of parabolic $R$-polynomials uniquely.

Koszul-Ringel duality relates these two families of polynomials as follows:

\begin{proposition}\label{prop7.4-7}
For all $x,y\in \mathbf{L}^\mathfrak{p}_{\mathrm{short}}$, we have 
$pr_{x,y}(v)=sr_{x^{-1},y^{-1}}(-v^{-1})$. 
\end{proposition}

\begin{proof}
Taking into account that the conjugation by $w_0$ is an automorphism of 
the Dynkin diagram, the claim of the proposition follows 
from the definitions using that the 
Koszul-Ringel duality sends $\Delta(w\cdot\lambda)$ to 
$\Delta^\mathfrak{p}_{w_0w^{-1}w_0}$
and $\nabla(w\cdot \lambda)$ to $\nabla_{w_0w^{-1}w_0}^\mathfrak{p}$
and intertwines $[i]\langle j\rangle$ with $[i+j]\langle -j\rangle$. 
\end{proof}

\subsection{Delorme formulae}\label{s7.5}

Similarly to the regular case, we have both ungraded and graded 
versions of Delorme formulae for both, the singular and the parabolic
cases, with the same proofs as for the regular case.

\begin{proposition}\label{prop7.5-1}
{\hspace{1mm}}

\begin{enumerate}[$($a$)$]
\item\label{prop7.5-1.1} 
For $x,y\in \mathbf{R}^\mathfrak{p}_{\mathrm{short}}$, we have
\begin{displaymath}
\sum_{i\geq 0}(-1)^i\dim\mathrm{Ext}^i_{\mathcal{O}}
(\Delta(x\cdot\lambda),\Delta(y\cdot\lambda))=\delta_{x,y}. 
\end{displaymath}
\item\label{prop7.5-1.2} 
For $x,y\in \mathbf{L}^\mathfrak{p}_{\mathrm{short}}$, we have
\begin{displaymath}
\sum_{i\geq 0}(-1)^i\dim\mathrm{Ext}^i_{\mathcal{O}^\mathfrak{p}}
(\Delta_x^\mathfrak{p},\Delta_y^\mathfrak{p})=\delta_{x,y}. 
\end{displaymath}
\end{enumerate}
\end{proposition}

\begin{proposition}\label{prop7.5-2}
{\hspace{1mm}}

\begin{enumerate}[$($a$)$]
\item\label{prop7.5-2.1} 
For $x,y\in \mathbf{R}^\mathfrak{p}_{\mathrm{short}}$ and $k\in\mathbb{Z}$, we have
\begin{displaymath}
\sum_{i\geq 0}(-1)^i\dim\mathrm{ext}^i(\Delta(x\cdot\lambda)
\langle k\rangle,\Delta(y\cdot\lambda))=sr_{x,y}^{(k)}. 
\end{displaymath}
\item\label{prop7.5-2.2} 
For $x,y\in \mathbf{L}^\mathfrak{p}_{\mathrm{short}}$ and $k\in\mathbb{Z}$, we have
\begin{displaymath}
\sum_{i\geq 0}(-1)^i\dim\mathrm{ext}^i_{(\mathcal{O}^\mathfrak{p})^\mathbb{Z}}
(\Delta_x^\mathfrak{p}\langle k\rangle,\Delta_y^\mathfrak{p})=pr_{x,y}^{(k)}. 
\end{displaymath}
\end{enumerate}
\end{proposition}

\subsection{Expected and additional extensions}\label{s.singexpected}

Under the indexing conventions in Subsections~\ref{s7.1},~\ref{s7.2}, the \emph{expected} extensions between singular or parabolic Verma modules are the extensions in 
\[\mathrm{ext}^k(\Delta(x\cdot \lambda),\Delta(y\cdot\lambda)\langle j\rangle),\quad 
\mathrm{ext}^{k+j}(\Delta^\mathfrak{p}_{x},
\Delta^\mathfrak{p}_{y}\langle -j\rangle)\]
where $2k+j = \ell(x)-\ell(y)$. The other (nonzero) extensions are \emph{additional}.
If all extensions are expected, then Proposition~\ref{prop7.5-2} says that the dimensions of the expected ext spaces are given by the $R$-polynomials.

\subsection{Koszulity}\label{s7.8}

Consider the categoy $\mathcal{D}^b(\mathcal{O}_\lambda^\mathbb{Z})$, for a dominant 
and integral (but not necessarily regular) $\lambda$. Denote by $\mathscr{D}$ 
the full subcategory of $\mathcal{D}(\mathcal{O}_\lambda^\mathbb{Z})$ 
given by the objects $\Delta(w\cdot\lambda)\langle i\rangle[j]$, where
$w\in \mathbf{R}^\mathfrak{p}_{\mathrm{short}}$ and $i,j\in\mathbb{Z}$ such that 
\begin{displaymath}
i\in\{-\ell(w)-2j,-\ell(w)-2j-1\}. 
\end{displaymath}
Then the same argument as Theorem~\ref{thmKoszul} gives the following generalization.

\begin{theorem}\label{thm7.8}
Let $\lambda$ be dominant and integral. Assume that all extensions between the Verma 
modules in $\mathcal{O}_\lambda$ are expected. Then the following assertions hold:
\begin{enumerate}[$($a$)$]
\item We have an equivalence
$\mathcal{D}^*(\mathcal{O}_\lambda^\mathbb{Z})\cong \mathcal{D}^*(\mathscr{D}\text{-}\mathrm{mod})$ where $*\in\{b,\uparrow,\downarrow\}$.
\item The path algebra of $\mathscr{D}$ is Koszul and is Koszul self-dual.
\end{enumerate}
\end{theorem}

\vspace{2mm}

\noindent
H.~K.: Department of Mathematics, Uppsala University, Box. 480,
SE-75106, Uppsala,\\ SWEDEN, email: {\tt hankyung.ko\symbol{64}math.uu.se}

\noindent
V.~M.: Department of Mathematics, Uppsala University, Box. 480,
SE-75106, Uppsala,\\ SWEDEN, email: {\tt mazor\symbol{64}math.uu.se}

\end{document}